\documentclass[jacodes,PDF]{jac} 
\usepackage{layout}

\newcommand{\deqno}{\refstepcounter{theorem}(\thetheorem)}
\theoremstyle{definition}
\newtheorem{disc}[theorem]{Discussion}
\newtheorem{ques}[theorem]{Question}

\newcounter{item}
\newenvironment{nlist}[1][1]{\begin{list}
  {\textup{(\arabic{item})}}{\usecounter{item} \setcounter{item}{#1}\addtocounter{item}{-1}
  \setlength{\itemsep}{0ex}
  \setlength{\topsep}{0ex} \setlength{\parsep}{0ex} \setlength{\labelwidth}{15mm}
  \setlength{\leftmargin}{10mm} } }{\end{list}}

\newcommand{\ds}{\displaystyle}
\newcommand{\zz}{\mathbb{Z}}

\title{A study of $m$-ary partitions whose conjugates are $q$-ary
}

\articletype{Research Article} 
\year{2026}
\issue{ 13(3)}
\ilksayfa{385} 
\shortauthor{G. D. Dietz et al.}  
\doi{\href{https://doi.org/10.13069/jacodesmath.v13i3.393}{10.13069/jacodesmath.v13i3.393}}
\gel{01 July 2025}   
\ok{07 April 2026}  

\author{
\href{https://orcid.org/0009-0002-4711-5174}{Geoffrey D. Dietz}\blfootnote{Geoffrey D. Dietz
(Corresponding Author); 
Department of Mathematics, Gannon University, Erie, PA, USA  (email: dietz005@gannon.edu)}, 
\href{https://orcid.org/0009-0001-0677-9559}{Timothy B. Flowers}\blfootnote{Timothy B. Flowers; Department of Mathematical and Computer Sciences, Indiana University of Pennsylvania, Indiana, PA, USA (email: flowers@iup.edu)},  
\href{https://orcid.org/0009-0002-0081-5165}{Shannon R. Lockard}\blfootnote{Shannon R. Lockard; Department of Mathematics, Bridgewater State University, Bridgewater, MA, USA (email: Shannon.Lockard@bridgew.edu)}
 }

\abstract{While people have studied $m$-ary partitions of an integer $n$ and studied conjugation of partitions of $n$, these topics are rarely mixed because the $m$-ary property is almost always lost after conjugation. In a previous work, Flowers and Lockard investigated $m$-ary partitions of $n$ whose conjugates were also $m$-ary. We generalize that previous work by studying $m$-ary partitions whose conjugates are $q$-ary, where $m$ and $q$ may be distinct. We provide a family of operators on these partitions that can be used to generate all such partitions uniquely and associate a unique polynomial with each partition based on the sequence of operators used to generate it. Using the generating operators and modular arithmetic we explore many examples and families of $m$-ary partitions whose conjugates are $q$-ary. }



\begin{document}
\maketitle

\begin{textblock*}{18cm} (8.3cm,26.55cm)
\begin{flushleft} \center{\small{ISSN 2148-838X}}
\end{flushleft}
\end{textblock*}
\begin{textblock*}{18cm} (2.6cm,26.55cm)
\begin{flushleft} \small{\doino}
\end{flushleft}
\end{textblock*}
\begin{textblock*}{18cm} (2.5cm,1.75cm)
\begin{flushleft} \textbf{\footnotesize{J. Algebra Comb. Discrete Appl.\\
\issuevol\;$\bullet$\;\ilkshf--\sonshf  }}
\end{flushleft}
\end{textblock*}
\begin{textblock*}{18cm} (.5cm,1.75cm)
\begin{flushright} \textbf{\footnotesize{Received: \gelis \\Accepted: \kabul}}
\end{flushright}
\end{textblock*}

\blfootnote{© 2026 The Author(s). Published by iPeak Academy Ltd. This is an open access article under the CC BY 4.0 license
(https://creativecommons.org/licenses/by/4.0/).}

\section{Introduction}

In this paper we present some new results related to $m$-ary partitions and to their conjugates. An $m$-ary partition of an integer $n$ is a partition of $n$ where the parts are all powers of $m$. These types of partitions have been extensively studied in the past. See the work of Mahler \cite{M40}, de Bruijn \cite{deB48}, and Pennington \cite{P53}. Additionally, see \cite{A71}, \cite{AFS15}, \cite{CW00}, \cite{C69}, \cite{CS04}, \cite{E18}, \cite{FL17}, \cite{G72}, and \cite{R70}. On the other hand, conjugation of partitions has played a long and prominent role in the field of integer partitions. See, for example, \cite{A84} and \cite{AE04}. 

These two ideas rarely intersect as conjugations of $m$-ary partitions are rarely $m$-ary.  In \cite{FL21}, Flowers and Lockard  investigated $m$-ary partitions whose conjugates are also $m$-ary using the following notation.


\begin{definition}Let $\lambda$ be a partition of the integer $n$.
\begin{nlist}
\item The \textbf{conjugate} of $\lambda$ will be the partition $\lambda'$ formed by reflecting its Ferrers diagram over its main diagonal. 
\item \cite{FL21} If all of the parts of $\lambda$ are powers of $m$, then we say $\lambda$ is \textbf{$m$-ary}. Let $a_i$ denote the number of times that the part $m^i$ appears in $\lambda$. Then $n = \sum_{i=0}^k a_i m^i$, and we write
$$
\lambda = (a_k,a_{k-1},\ldots, a_1, a_0)_m
$$
and refer to each $a_i$ as the \textbf{multiplicity} of part $m^i$. 
\item If $\lambda$ and $\lambda'$ are both $m$-ary, then $\lambda$ is a \textbf{conjugate $m$-ary partition} (or \textbf{CMP}). 
\end{nlist}
\end{definition}

For example, $\lambda = (1,2,0)_3$ with parts $9,3,3$ is a $3$-ary CMP of $15$ as $\lambda' = (3,6)_3$ with parts $3,3,3,1,1,1,1,1,1$. As another example, $\lambda = (16, 0)_4$ is a $4$-ary CMP of $64$ with $\lambda' = (4,0,0)_4$. We call these latter examples rectangular because of the shape of their respective Ferrers diagrams.

In the rest of this article we will extend the work of \cite{FL21} to $m$-ary partitions whose conjugates are $q$-ary, where $m$ and $q$ need not be the same base. 

\begin{definition}
Let $\lambda$ be a partition of an integer $n$. 
$\lambda$ is called \textbf{conjugate $(m,q)$-ary} if $\lambda$ is $m$-ary while $\lambda'$ is $q$-ary. 
As in  \cite{FL21}, we denote conjugate $(m,q)$-ary partitions as $\lambda = (a_k,a_{k-1},\ldots,a_0)_m$ meaning that the partition has the part $m^i$ appearing $a_i$ times. 
\end{definition}

For example, $\lambda = (2,6)_3$ is a $3$-ary partition of 12 with parts $3,3,1,1,1,1,1,1$ while $\lambda' = (1,0,2,0)_2$ is a $2$-ary partition of 12 with parts $8,2,2$. 

By computing partial sums of the multiplicities in $\lambda$ and following the algorithm for conjugation, one can find the parts of $\lambda'$. If those partial sums are all powers of $q$, then $\lambda'$ will be $q$-ary. We state this result below. The argument is essentially given in \cite[Theorem 2.2]{FL21}. 

\begin{proposition}
If $\lambda$ is an $m$-ary partition, then it is a conjugate $(m,q)$-ary partition if and only if the partial sums $\sum_{i=r}^k a_i$ are powers of $q$ for all $0\leq r\leq k$.
\end{proposition}

In the next section, we introduce a family of operators $R,M,S$, and $J$ that act on the set of conjugate $(m,q)$-ary partitions. We will use those operators to generate the set of all conjugate $(m,q)$-ary partitions in a unique manner based on sequences of $R$ and $M$ or sequences of $S$ and $J$. See Theorem \ref{repthm}. We can also show that if one reverses the order of the operators used to generate $\lambda$, then one generates its conjugate $\lambda'$.  See Corollary \ref{conjcor}. The pairs of operators $R, M$ and $S,J$ have a duality relationship (see Lemma \ref{conjlemma}) that allows us to focus on just one of the pairs of operators. For the most part, we choose to use the $R,M$ pair for our later results, but similar results would also follow using $S,J$ instead. 

In the third section, we associate a unique polynomial to each sequence of the operators $R$ and $M$. See Definition \ref{polyndef}. With the aid of induction, we then can calculate an explicit expression for each such polynomial. See Theorem \ref{polynthm}. (Warning: the polynomials are rather complicated and not for the faint of heart!) These polynomials give a means of classifying whether or not a given integer $n$ has a conjugate $(m,q)$-ary partition. See Corollary \ref{nclassify}. These results then help place bounds on the possible $n$ values that have such partitions or on the number of operators $R$ and $M$ that can generate a conjugate $(m,q)$-ary partition of such an $n$. 

In the fourth section, we use modular arithmetic to better understand which integers $n$ can have a conjugate $(m,q)$-ary partition for a given pair of integers $m$ and $q$. In some cases, we can also count how many such partitions are possible. The operators $M$ and $R$ are again of great utility in this work. We provide some general results and examples.

In the fifth section, we use the $(m,q)$-ary techniques developed here in order to further study the CMPs of \cite{FL21}, particularly CMPs of powers of $m$. We use all four operators $R,M,S$, and $J$ to calculate families of CMPs for certain powers of $m$, where $m$ is any integer greater than or equal to 2. These results are summarized in Discussion \ref{summarylist} and followed by some interesting open questions related to CMPs of powers of $m$. 

In the final section, we investigate how conjugate $(m,q)$-ary partitions are related to $(m^a,q^b)$-ary partitions. We end by looking at the case where $m$ and $q$ are both powers of a common root $w$. In Theorem \ref{commonpowerthm} we generalize \cite[Theorem 3.4]{FL21} by counting exactly how many conjugate $(w^a,w^b)$-ary partitions an integer $n$ can have if $n\equiv -1 \pmod{w}$.

Along the way we will also draw attention to many of the highlights of \cite{FL21} and see how those results can be generalized to our new setting. 

\section{Generating Operators}

In \cite{FL21}, two operators, raise and shift, were introduced as methods to generate new CMPs given a CMP. We extend these notions to conjugate $(m,q)$-ary partitions and introduce two other operations that we call migrate and jump. 

\begin{definition} \label{opdef}
Let $\lambda = (a_k,a_{k-1},\ldots, a_0)_m$ be a conjugate $(m,q)$-ary partition of the integer $n$ with $\sum_{i=0}^k a_i = q^t$.
\begin{nlist}
\item The \textbf{raise} of $\lambda$ is $R(\lambda) =  (qa_k,qa_{k-1},\ldots, qa_0)_m$, which is  a conjugate   $(m,q)$-ary  partition of $qn$.
\item The \textbf{migration} of $\lambda$ is $M(\lambda) =  (1,a_k-1,a_{k-1},\ldots, a_0)_m$, which is  a conjugate   $(m,q)$-ary  partition of $n + m^{k+1}-m^k$. 
\item The \textbf{shift} of $\lambda$ is $S(\lambda) =  (a_k,a_{k-1},\ldots, a_0,0)_m$, which is  a conjugate   $(m,q)$-ary  partition of $mn$. 
\item The \textbf{jump} of $\lambda$ is $J(\lambda) =  (a_k,a_{k-1},\ldots, a_1, a_0+q^{t+1}-q^t)_m$, which is  a conjugate  $(m,q)$-ary  partition of $n + q^{t+1}-q^t$. 
\end{nlist}
In each case, if there is ambiguity about the type of partition (e.g., Lemma~\ref{conjlemma}, Theorem~\ref{repthm}, and Corollary~ \ref{conjcor}), we use subscripts such as $R_{mq}(\lambda)$ to indicate that raise is being applied to  a conjugate $(m,q)$-ary partition or $R_{qm}(\lambda')$ to indicate raise is applied to the conjugate $(q,m)$-ary partition $\lambda'$. If no subscripts are displayed in a statement or proof, the reader may safely assume that all operators shown have subscripts of $mq$ and are acting only on conjugate $(m,q)$-ary partitions. 
\end{definition}

\begin{remark} \label{usefulremark} We list a few easy to see but useful facts about these operators. 
\begin{nlist}
\item Each of $R$, $M$, $S$, and $J$ is a one-to-one operator acting on the set of $(m,q)$-ary partitions. 
\item $R^s(q^t)_m = (q^{t+s})_m = (q^{t} + q^{t+s}-q^t)_m = J^s(q^t)_m$ 
\item $M^r(1)_m = (1,0,\ldots,0)_m = S^r(1)_m$
\end{nlist}
\end{remark}

The operators $R$ and $M$ are dual to $S$ and $J$, respectively, via the conjugation operation as illustrated in the following lemma. 

\begin{lemma} \label{conjlemma} Let $\lambda = (a_k,a_{k-1},\ldots, a_0)_m$ be a conjugate $(m,q)$-ary partition of the integer $n$ with $a_k \neq 0$, and let $\lambda' =  (b_\ell,b_{\ell-1},\ldots, b_1, b_0)_q$ be the conjugate $(q,m)$-ary partition. Then 
\begin{nlist}
\item $R_{mq}(\lambda)' = S_{qm}(\lambda')$ and $S_{mq}(\lambda)' = R_{qm}(\lambda')$
\item $M_{mq}(\lambda)' = J_{qm}(\lambda')$ and $J_{mq}(\lambda)' = M_{qm}(\lambda')$
\end{nlist}
\end{lemma}
\begin{proof}
The proof of (1) for conjugate $m$-ary partitions is given in \cite{FL21}. The proof for conjugate $(m,q)$-ary partitions is essentially the same. In a conjugate $(m,q)$-ary partition $\lambda = (a_k,a_{k-1},\ldots, a_0)_m$, the partial sums of the multiplicities $\sum_{i=r}^k a_i$ must be powers of $q$. As
$R_{mq}(\lambda) = (qa_k,qa_{k-1},\ldots, qa_0)_m$,
each partial sum in $R_{mq}(\lambda)$ is now increased by one power of $q$ compared to $\lambda$. The conjugation algorithm \cite[p461]{FL21} applied to $\lambda$ leads to writing $n$ as a sum of terms that are differences of powers of $m$ multiplied by $q^r$ for some $r$. Consequently, each $b_r$ in $\lambda'$ is the difference of powers of $m$. Then conjugating $R_{mq}(\lambda)$ will raise each power $q^r$ to $q^{r+1}$ and shift each $b_r$ of $\lambda'$ one place higher leading to
$$
R_{mq}(\lambda)' = (b_\ell,b_{\ell-1},\ldots, b_1, b_0,0)_q = S_{qm}(\lambda').
$$
Swapping the roles of $\lambda$ and $\lambda'$ gives the other result in (1).

For (2), apply the conjugation algorithm  \cite[p461]{FL21} to 
$$M_{mq}(\lambda) =  (1,a_k-1,a_{k-1},\ldots, a_0)_m.$$ 
The partial sums of the multiplicities are $1 +(a_k -1) + \sum_{i=r}^{k-1} a_i = \sum_{i=r}^k a_i$, just as in conjugating $\lambda$ except the new top multiplicity at position $k+1$, which is just 1. Hence the only difference in the conjugate expression for $M_{mq}(\lambda)'$ versus $\lambda'$ is an extra $(m^{k+1}-m^k)\cdot 1$ at the end, and so
$$
M_{mq}(\lambda)' =  (b_\ell,b_{\ell-1},\ldots, b_1, b_0 + m^{k+1}-m^k)_q
$$
On the other hand, since $\lambda$ is the conjugate of $\lambda'$ and $a_k\neq 0$, the sum $\sum_{j=0}^\ell b_j = m^k$. Therefore,
$$
J_{qm}(\lambda') = (b_\ell,b_{\ell-1},\ldots, b_1, b_0 + m^{k+1}-m^k)_q = M_{mq}(\lambda)'.
$$
Again, swapping roles for $\lambda$ and $\lambda'$ gives the other result. 
\end{proof}

In the next Lemma we show that the operators commute with each other across the dual pairs.

\begin{lemma}\label{commlemma} Let $\lambda = (a_k,a_{k-1},\ldots, a_0)_m$ be an $(m,q)$-ary partition of integer $n$ with $\sum_{i=0}^k a_i = q^t$.
\begin{nlist}
\item $RS = SR$. 
\item $MS = SM$.
\item $RJ = JR$.
\item $MJ = JM$. 
\end{nlist}
\end{lemma}
\begin{proof} We show each pairing has the same end result when applied to $\lambda$.
\begin{nlist}
\item It is straightforward to check that $RS(\lambda)$ and $SR(\lambda)$ result in 
$$(qa_k,qa_{k-1},\ldots, qa_0,0)_m.$$
\item Note that $MS(\lambda)$ and $SM(\lambda)$ result in $ (1,a_k-1,a_{k-1},\ldots, a_0,0)_m$.
\item We have 
$$
\begin{array}{rcl}
RJ(\lambda)  &= & R(a_k,a_{k-1},\ldots, a_1, a_0+q^{t+1}-q^t)_m \\
& = & (qa_k,qa_{k-1},\ldots, qa_1, qa_0+q^{t+2}-q^{t+1})_m
\end{array}
$$ 
while 
$$
\begin{array}{rcl}
JR(\lambda)  &= & J(qa_k,qa_{k-1},\ldots,  qa_0)_m \\
& = & (qa_k,qa_{k-1},\ldots, qa_1, qa_0+q^{t+2}-q^{t+1})_m
\end{array}
$$ 
because $R(\lambda)$ has $\sum_{i=0}^k q a_i = q^{t+1}$.
\item Note that $\sum_{i=0}^k a_i = q^t$ means that $1 + (a_k-1) + \sum_{i=0}^{k-1} a_i = q^t$ as well.  Thus, $MJ(\lambda)$ and $JM(\lambda)$ result in
$$
 (1,a_k - 1,a_{k-1},\ldots, a_1, a_0+q^{t+1}-q^t)_m 
$$
\end{nlist}
\end{proof}

On the other hand $R$ and $M$ do not commute nor do $S$ and $J$.

\begin{example} Let $m,q\geq 2$.  Again, let $\lambda = (a_k,a_{k-1},\ldots, a_0)_m$ be an $(m,q)$-ary partition of the integer $n$ with $\sum_{i=0}^k a_i = q^t$.
\begin{nlist}
\item For $R$ and $M$ note that 
$$
RM(\lambda) = R(1,a_k-1,a_{k-1},\ldots, a_0)_m = (q,qa_k-q,qa_{k-1},\ldots, qa_0)_m,
$$
which is a partition of $qn + qm^{k+1} - qm^k$ while
$$
MR(\lambda) = M(qa_k,qa_{k-1},\ldots, qa_0)_m = (1,qa_k - 1,qa_{k-1},\ldots, qa_0)_m,
$$
which is a partition of $qn+m^{k+1}-m^k$. Thus $RM \neq MR$. 
\item For $S$ and $J$ note that
\begin{eqnarray*}
    SJ(\lambda) &=& S(a_k,a_{k-1},\ldots, a_1, a_0+q^{t+1}-q^t)_m \\
    &=& (a_k,a_{k-1},\ldots, a_1, a_0+q^{t+1}-q^t,0)_m,    
\end{eqnarray*}
which is a partition of  $mn + mq^{t+1}-mq^t$ while
$$
JS(\lambda) = J(a_k,a_{k-1},\ldots, a_1, a_0,0)_m = (a_k,a_{k-1},\ldots, a_1, a_0,q^{t+1}-q^t)_m
$$
which is a partition of $mn + q^{t+1}-q^t$. Thus $SJ\neq JS$.
\end{nlist}
\end{example}

We are almost ready to prove that each conjugate $(m,q)$-ary partition is uniquely generated by a sequence of either $R$ and $M$ operators or by a sequence of $S$ and $J$ operators. We establish a lemma first.

\begin{lemma} \label{dividelemma}
Let $\lambda = (a_k,a_{k-1},\ldots, a_0)_m$ be a conjugate $(m,q)-ary$ partition. Then $a_k$ divides $a_j$ for all $j$. 
\end{lemma}
\begin{proof}
First note that $a_k = q^{t_0}$ for some $t_0$ and $a_k+a_{k-1} = q^{t_1}$ with $t_1\geq t_0$. Thus, $a_k = q^{t_0}$ divides $q^{t_1}$ and so divides $a_{k-1}$ as a result. As partial sums $a_k + \cdots + a_j$ form a non-decreasing sequence of powers of $q$, we see that $a_k$ divides all $a_j$.
\end{proof}

Now we have the tools for the theorem.

\begin{theorem} \label{repthm}
Let $m,q\geq 2$. Let $\lambda = (a_k,a_{k-1},\ldots, a_0)_m$ be a conjugate $(m,q)-ary$ partition. Then $\lambda$ is generated by a unique sequence of $R_{mq}$ and $M_{mq}$ operators applied to the trivial partition $(1)_m$ of 1. Also, $\lambda$ is generated by a unique sequence of $S_{mq}$ and $J_{mq}$ operators.
\end{theorem}
\begin{proof}
The existence proof uses induction on the number of multiplicities $k$ in $\lambda$. We focus on the $R_{mq}$ and $M_{mq}$ operators for now. If $k=0$, then $\lambda = (a_0)_m$, where $a_0 = q^t$ for some $t$. Thus, $\lambda = R_{mq}^t(1)_m$. Suppose $k\geq 0$ and that all conjugate $(m,q)$-ary partitions of the form $(a_k,a_{k-1},\ldots, a_0)_m$ are generated by a sequence of $R_{mq}$ and $M_{mq}$ operators applied to $(1)_m$. Let $\lambda = (a_{k+1},a_k,a_{k-1},\ldots, a_0)_m$ be a conjugate $(m,q)$-ary partition. Then $a_{k+1} = q^t$ for some $t\geq 0$ and $a_{k+1} + a_{k} = q^s$ for $s\geq t$. Thus, 
$$
a_k = q^s - q^t = q^t(q^{s-t} - 1)
$$
By Lemma \ref{dividelemma}, all $a_j$ are divisible by $q^t$ so that 
$$
\lambda = R_{mq}^t(1,q^{s-t}-1, b_{k-1}, \ldots, b_0)_m = R_{mq}^t M_{mq}(q^{s-t},b_{k-1},\ldots,b_0)_m.
$$ 
By the inductive hypothesis, the partition $(q^{s-t},b_{k-1},\ldots,b_0)_m$ is generated by a sequence of $R_{mq}$ and $M_{mq}$ operators so that $\lambda$ is as well. 

To generate such a partition using $S_{mq}$ and $J_{mq}$ operators, we first generate the conjugate partition $\lambda'$ from a sequence of $R_{qm}$ and $M_{qm}$ operators by the first part of the proof. Applying conjugation to that sequence and repeatedly using Lemma \ref{conjlemma} results in a sequence of $S_{mq}$ and $J_{mq}$ operators that generate $\lambda$. Since $(1)_m' = (1)_q = (1)_m$, we are still starting with the trivial partition of 1. 

To prove that these sequences are unique, suppose that $A_1$ and $A_2$ are different sequences of either $R_{mq}$ and $M_{mq}$ operators (or $S_{mq}$ and $J_{mq}$ operators) such that $A_1(1)_m = A_2(1)_m$. Using the one-to-one property (Remark \ref{usefulremark}), we can assume without loss of generality that the leftmost operators in $A_1$ and $A_2$ differ from each other. For $R_{mq}$ and $M_{mq}$, one partition would have $q^t$, $t>0$, as its top multiplicity while the other would have $1$, a contradiction. For $S_{mq}$ and $J_{mq}$, one partition would have $0$ as its rightmost multiplicity while the other would be at least $q^{t+1}-q^t$ for $t\geq 0$, another contradiction. 
\end{proof}

Given the previous theorem, we can focus on one pair of these generating operators. We choose to focus on $R$ and $M$. 

\begin{example} \label{tree}
The tree below illustrates how all conjugate $(5,3)$-ary partitions can be eventually generated using the $M= M_{53}$ and $R = R_{53}$ operators. 
Looking ahead toward later sections, it is worth noting that $(1,0,2)_5 = M^2R(1)_5$ and $(27)_5 = R^3(1)_5$ both generate partitions of $27$. Counting how many conjugate $(m,q)$-ary partitions a given integer $n$ has (if any) appears to be very challenging and relates to subtle modular arithmetic relationships between $m$ and $q$. 
\end{example}
\begin{figure}[h]
\centering
\includegraphics[scale=0.7]{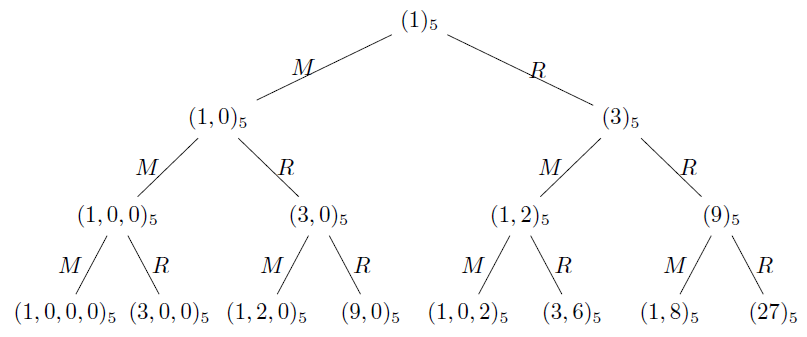}
\caption{Start of the tree for all conjugate $(5,3)$-ary partitions generated by $M$ and $R$ operators}
\end{figure}

While the operators $R$ and $M$ do not commute with each other, we can quantify how many total copies of $R$ or $M$ occur in a generating sequence for a conjugate $(m,q)$-ary partition.

\begin{corollary} \label{opcount}
Let $m,q\geq 2$. Let $\lambda = (a_k,a_{k-1},\ldots, a_0)_m$ be a conjugate $(m,q)$-ary partition with $a_k\neq 0$ and $\sum_{i=0}^k a_i = q^t$. Then the generating sequence for $\lambda$ contains 
\begin{nlist}
\item $k$ operations of $M$, and
\item $t$ operations of $R$
\end{nlist}
in some order. 
\end{corollary}
\begin{proof}
The $R$ operator changes the sizes of the multiplicities but only the $M$ operator can add an additional multiplicity. Starting from $(1)_m$, we need $k$ operations of $M$ to produce a partition with $a_k \neq 0$. Next, we note that when $\sum_{i=0}^k a_i = q^t$, then either $a_0 = q^t$ or $a_0 = q^t - q^s = q^s(q^{t-s}-1)$ for some $s$. In the first case, $\lambda = (q^t)_m = R^t(1)_m$. In the second case, observe that $MR^{t-s}(1)_m = (1,q^{t-s}-1)_m$. Since generating sequences are unique, we see that the generating sequence for $\lambda$ must have started with $t-s$ operations of R followed by one M. From then on, $M$ operators have no effect on the rightmost position while $s$ operations of $R$ will be needed to get to $a_0 =  q^s(q^{t-s}-1)$. In either case $t$ operations of $R$ are needed. 
\end{proof}

Using these sequences of operators, we can classify conjugate $(m,q)$-ary partitions including the notion of simple partitions as used in \cite{FL21}.

\begin{definition}\cite[Definition 6.1]{FL21}
 A conjugate $(m,q)$-ary partition is \textbf{simple} if and only if it is neither a raise nor a shift of another such partition. 
\end{definition}

\begin{corollary} \label{repcor}
If $\lambda$ is a conjugate $(m,q)$-ary partition, then $\lambda$ can be written uniquely as $\lambda = M^{s_k}R^{t_k}\cdots M^{s_1}R^{t_1}(1)_m$, where $s_k, t_1 \geq 0$ but all other exponents are strictly positive. Moreover, $\lambda$ is simple if and only if $s_k, t_1 > 0$
\end{corollary}
\begin{proof}
The first claim follows directly from Theorem \ref{repthm}. For the second claim note that $\lambda$ is the raise of another partition if and only if $s_k = 0$ so that its generating sequence has an $R$ on the far left. Similarly, $t_1 = 0$ if and only if the sequence for $\lambda$ starts with $M$. By Remark \ref{usefulremark}, that is equivalent to the same sequence but replacing the initial $M$ with an $S$. By Lemma \ref{commlemma}, that $S$ can commute to the left-most position of the generating sequence. Thus, $t_1= 0$ if and only if $\lambda$ can be written as a shift of another partition. 
\end{proof}

We can also understand the conjugation operation in a new way. Essentially, running the exponent sequence in the reverse order leads to the conjugate. 

\begin{corollary} \label{conjcor}
Let $\lambda = M_{mq}^{s_k}R_{mq}^{t_k}\cdots M_{mq}^{s_1}R_{mq}^{t_1}(1)_m$ be a conjugate $(m,q)$-ary partition. Then $\lambda' =  M_{qm}^{t_1}R_{qm}^{s_1}\cdots M_{qm}^{t_k}R_{qm}^{s_k}(1)_q$ is its conjugate $(q,m)$-ary partition. 
\end{corollary}
\begin{proof}
We use induction on the sum $s_k+t_k+\cdots + s_1+t_1$. If the sum is 1, then either $t_1 = 1$ or $s_1 = 1$. In the first case, $\lambda = R_{mq}(1)_m$. By Lemma \ref{conjlemma} and Remark \ref{usefulremark}, $\lambda' = S_{qm}(1)_q = M_{qm}(1)_q$. If $\lambda = M_{mq}(1)_m$, then $\lambda' = J_{qm}(1)_q = R_{qm}(1)_q$. Suppose that the result is true when the sum of exponents is $s_k+t_k+\cdots + s_1+t_1$. Consider $\lambda$ that is generated by one additional operator. We can then assume that either $\lambda = R_{mq} M_{mq}^{s_k}R_{mq}^{t_k}\cdots M_{mq}^{s_1}R_{mq}^{t_1}(1)_m$  or $\lambda = M_{mq}^{s_k+1}R_{mq}^{t_k}\cdots M_{mq}^{s_1}R_{mq}^{t_1}(1)_m$ (where $s_k$ might be 0). Then applying Lemma \ref{conjlemma} and the inductive hypothesis, we have $\lambda'$ is either $S_{qm}M_{qm}^{t_1}R_{qm}^{s_1}\cdots M_{qm}^{t_k}R_{qm}^{s_k}(1)_q$ or $J_{qm}M_{qm}^{t_1}R_{qm}^{s_1}\cdots M_{qm}^{t_k}R_{qm}^{s_k}(1)_q$. We then use Lemma \ref{commlemma} to commute the $S$ or $J$ to the far right giving $M_{qm}^{t_1}R_{qm}^{s_1}\cdots M_{qm}^{t_k}R_{qm}^{s_k}S_{qm}(1)_q$ or $M_{qm}^{t_1}R_{qm}^{s_1}\cdots M_{qm}^{t_k}R_{qm}^{s_k}J_{qm}(1)_q$. Finally use Remark \ref{usefulremark} to trade out the $S$ or $J$ in the front now as an $M$ or $R$, giving $\lambda' = M_{qm}^{t_1}R_{qm}^{s_1}\cdots M_{qm}^{t_k}R_{qm}^{s_k}M_{qm}(1)_q$ or $\lambda' = M_{qm}^{t_1}R_{qm}^{s_1}\cdots M_{qm}^{t_k}R_{qm}^{s_k+1}(1)_q$. 
\end{proof}

Looking at the CMPs  of \cite{FL21}, we can easily detect self-conjugate $(m,m)$-ary partitions. In the statement below all operators could be written with $mm$ subscripts, but we omit them to simplify the statement. 

\begin{corollary} \label{selfconj}
Let $\lambda = M^{s_k}R^{t_k}\cdots M^{s_1}R^{t_1}(1)_m$ be a conjugate $(m,m)$-ary partition, where $s_k, t_1 \geq 0$ but all other exponents are strictly positive. Then $\lambda = \lambda'$ if and only if the sequence $(s_k,t_k,\ldots,s_1,t_1)$ is a palindrome. 
\end{corollary}

\section{Characteristic Polynomials}

Given that all conjugate $(m,q)$-ary partitions $\lambda$ are uniquely generated by sequences of $R$ and $M$ operators, we look closer at the number $n$ that is partitioned by $\lambda$. If we treat $m$ and $q$ as independent algebraic indeterminates, we can associate a unique polynomial function with each generating sequence and thus with each conjugate $(m,q)$-ary partition.

\begin{definition} \label{polyndef}
Let $m,q\geq 2$, and let $A$ denote a sequence of $R$ and $M$ operators. Then $\lambda = A(1)_m  = (a_k,a_{k-1},\ldots,a_0)_m$ is a conjugate $(m,q)$-ary partition, where $a_k = q^{t_k}$ and $a_i = q^{t_i} - q^{t_{i+1}}$.  The \textbf{characteristic polynomial} of $A$, denoted by $\chi(A)$, is
$$
\chi(A) = \sum_{i=0}^k a_i m^i = q^{t_k}m^k + \sum_{i=0}^{k-1} (q^{t_i} - q^{t_{i+1}})m^i
$$
Note that $\chi(A) \in \zz[m,q]$. 
\end{definition}

\begin{example}
Consider conjugate $(5,3)$-ary partitions as in Example~\ref{tree}. We saw that $M^2R(1)_5 = (1,0,2)_5$ and $R^3(1)_5 = (27)_5$ are both partitions of 27. When we view $m$ and $q$ as indeterminates, we see that 
$$
\chi(M^2 R) = 1\cdot m^2 + 2 m^0 = q^0 m^2 +  (q^0 - q^0) m + (q^1 - q^0) m^0 = m^2 + q - 1
$$
and
$$
\chi(R^3) = q^3 m^0 = q^3 ,
$$
where both polynomials give 27 when evaluated at $(m,q) = (5,3)$ but represent different conjugate $(5,3)$-ary partitions of 27.  The polynomials in terms of $m$ and $q$ are independent of the actual values of $m$ or $q$ chosen. 
\end{example}

These polynomials give a different way to classify the conjugate $(m,q)$-ary partitions, but they also give a more direct method of identifying the number $n$ being partitioned. If one evaluates $\chi(A)$ at a specific pair of values $(m,q)$, then the result is this number $n$. Our later attempts to discover cases where a common number $n$ has multiple conjugate $(m,q)$-ary partitions will stem from finding cases when different polynomials evaluate to the same value $n$ as in the previous example. 

\begin{example}
Consider a general pair $(m,q)$. As $M^2R^2(1)_m = (1,0,q^2-1)_m$, 
$$
\chi(M^2R^2) = m^2 + (q^2-1)m^0 = m^2+q^2-1.
$$
As $MRMR(1)_m = (1,q-1,q^2-q)_m$, we have
$$
\chi(MRMR) = m^2 + (q-1)m + (q^2-q)m^0 = m^2 - m + mq - q + q^2.
$$

\end{example}

With the help of the $R$ and $M$ operators, we can classify all such characteristic polynomials. The next lemma shows that studying simple partitions are sufficient as tacking $R$ operators at the end or $M$ at the beginning (equivalently, adding shifts at the end) results in simple multiplication by $q$ or $m$.

\begin{lemma}
Let $m,q\geq 2$. Let $A$ be any sequence of $R$ and $M$ operators, generating a conjugate $(m,q)$-ary partition $\lambda$. Then $\chi(R^tAM^s) = m^sq^t\chi(A)$.
\end{lemma}
\begin{proof}
$R^tAM^s(1)_m  = R^tAS^s(1)_m = S^sR^tA(1)_m$. The shift operators scale each term by $m$, and the raise operators scale each term by $q$.
\end{proof}

From Corollary \ref{repcor}, we know we can focus on sequences made by alternating blocks of powers of $M$ and $R$ operators as we classify the possible characteristic polynomials. Theorem \ref{polynthm} will give general closed-form formulas for all possible characteristic polynomials. The formulas may be a bit overwhelming at first and break down into odd and even cases. We start with the smallest two cases in the following lemmas and use those as base cases for an induction argument in the theorem. 

\begin{lemma}\label{oddbase}
Consider $s_1, t_1 \geq 0$ with at least one of them positive. Then $$\chi(M^{s_1}R^{t_1}) = q^{t_1} + m^{s_1} - 1 = (m^{s_1} - 1) + m^0q^0 + (q^{t_1} - 1).$$
\end{lemma}
\begin{proof}
By Definition \ref{opdef}, $M^{s_1}R^{t_1}(1)_m = M^{s_1}(q^{t_1})_m = (1,0,\ldots,0,q^{t_1}-1)_m$, with $s_1 -1$ zeros in the list if $s_1>0$. Thus,
$$
\chi(M^{s_1}R^{t_1}) = m^{s_1} + (q^{t_1}-1)m^0 = (m^{s_1} - 1) + m^0q^0 + (q^{t_1}-1).
$$
If $s_1 = 0$, then  $M^{s_1}R^{t_1}(1)_m = (q^{t_1})_m$ and so $$\chi(M^{s_1}R^{t_1}) = q^{t_1}m^0 = (m^{s_1} - 1) + m^0q^0 + (q^{t_1}-1).$$\end{proof}

\begin{lemma}\label{evenbase}
Let $s_1,s_2, t_1,t_2\geq 0$ be integers with $s_1,t_2>0$. Then 
$$
\chi(M^{s_2}R^{t_2}M^{s_1}R^{t_1})  = m^{s_1}(m^{s_2}-1) + m^{s_1}q^{t_2} + q^{t_2}(q^{t_1}-1).
$$
\end{lemma}
\begin{proof}
Working from the proof of the last lemma, 
$$
\begin{array}{rcl}
M^{s_2}R^{t_2}M^{s_1}R^{t_1}(1)_m &=& M^{s_2}R^{t_2}(1,0,\ldots,0,q^{t_1}-1)_m \\
&=& M^{s_2}(q^{t_2},0,\ldots,0,q^{t_2}(q^{t_1}-1))_m,
\end{array}
$$
which is  $(1,0,\ldots,0,q^{t_2}-1,0,\ldots,0,q^{t_2}(q^{t_1}-1))_m$ with $s_2-1$ zeros and $s_1-1$ zeros in the list if $s_2>0$ or $(q^{t_2},0\ldots,0,q^{t_2}(q^{t_1}-1))_m$ if $s_2 = 0$. Hence for $s_2 >0$ 
$$
\begin{array}{rcl}
\chi(M^{s_2}R^{t_2}M^{s_1}R^{t_1}) &=& m^{s_1+s_2} + (q^{t_2}-1)m^{s_1} + q^{t_2}(q^{t_1}-1) \\
&=& m^{s_1}(m^{s_2}-1) + m^{s_1}q^{t_2} + q^{t_2}(q^{t_1}-1).
\end{array}
$$
If $s_2=0$, then we have $
\chi(M^{s_2}R^{t_2}M^{s_1}R^{t_1}) = q^{t_2}m^{s_1} + q^{t_2}(q^{t_1}-1) = m^{s_1}(m^{s_2}-1) + m^{s_1}q^{t_2} + q^{t_2}(q^{t_1}-1).
$
\end{proof}

We break out this theorem into even and odd cases. The two formulas could be merged into a single formula based on floors and ceilings, but we elected to list them separately in the hopes of increased clarity.

\begin{theorem} \label{polynthm}
Let $m,q\geq 2$. Let $\lambda$ be a conjugate $(m,q)$-ary partition generated by $M^{s_k}R^{t_k}\cdots M^{s_1}R^{t_1}(1)_m$,  where $s_{k}, t_1 \geq 0$ but all other exponents are strictly positive. Suppose that $k = 2b+1$ is odd. Then  $\chi(M^{s_{2b+1}}R^{t_{2b+1}}\cdots M^{s_1}R^{t_1})$
$$ 
\begin{array}{c}
\ds = \sum_{i=1}^{b} \left( m^{s_1+\cdots+s_{2b+1-i}}q^{t_{2b+3-i}+\cdots+ t_{2b+1}}(m^{s_{2b+2-i}} -1)  \right) \\[5mm]
\ds + m^{s_1+\cdots+s_{b+1}}q^{t_{b+2}+\cdots+t_{2b+1}} - m^{s_1+\cdots+s_{b}}q^{t_{b+2}+\cdots+t_{2b+1}} + m^{s_1+\cdots+s_{b}}q^{t_{b+1}+\cdots+t_{2b+1}} \\[2mm]
\ds  + \sum_{i=1}^{b} \left(  m^{s_1+\cdots+s_{i-1}}q^{t_{i+1}+\cdots+ t_{2b+1}}(q^{{t_i}} -1)  \right)  
\end{array} \leqno \deqno \label{charpolyodd}
$$
Suppose that $k = 2b$ is even. Then $\chi(M^{s_{2b}}R^{t_{2b}}\cdots M^{s_1}R^{t_1})$
$$ 
\begin{array}{c}  
\ds = \sum_{i=1}^{ b } \left( m^{s_1+\cdots+s_{2b-i}}q^{t_{2b+2-i}+\cdots+ t_{2b}}(m^{s_{2b+1-i}} -1)  \right) \\[4mm]
\ds + m^{s_1+\cdots+s_{b }}q^{t_{ b + 1}+\cdots+t_{2b}} 
\ds + \sum_{i=1}^{ b } \left(  m^{s_1+\cdots+s_{i-1}}q^{t_{i+1}+\cdots+ t_{2b}}(q^{{t_i}} -1)  \right)  
\end{array}  \leqno \deqno \label{charpolyeven}
$$
In each case, if $\lambda$ is simple so that all exponents are strictly positive, then there are $k+1$ distinct monomials with coefficient $+1$ and $k$ distinct monomials with coefficient $-1$. 
\end{theorem}
\begin{proof}
The proof is by induction  on $k$. The odd and even base cases are given in the previous lemmas by setting $b=0$ (i.e., $k=1$) and $b=1$ (i.e., $k=2$), respectively. 
The induction then breaks down into even and odd cases as well. 

Consider an odd value $k = 2b+1$ and suppose that $\chi(M^{s_{2b}}R^{t_{2b}}\cdots M^{s_1}R^{t_1})$ has the formula as given in (\ref{charpolyeven}).
Then $\chi(M^{s_{2b+1}}R^{t_{2b+1}}\cdots M^{s_1}R^{t_1})$ is given by 
$$
\begin{array}{c}
\ds q^{t_{2b+1}} \sum_{i=1}^{ b } \left( m^{s_1+\cdots+s_{2b-i}}q^{t_{2b+2-i}+\cdots+ t_{2b}}(m^{s_{2b+1-i}} -1)  \right) \\[4mm]
\ds + q^{t_{2b+1}} m^{s_1+\cdots+s_{b }}q^{t_{ b + 1}+\cdots+t_{2b}} 
\ds + q^{t_{2b+1}}\sum_{i=1}^{ b } \left(  m^{s_1+\cdots+s_{i-1}}q^{t_{i+1}+\cdots+ t_{2b}}(q^{{t_i}} -1)  \right)  \\[4mm]
\ds + m^{s_1+\cdots+s_{2b+1}} - m^{s_1+\cdots+s_{2b}} 
\end{array}
$$
Distribute $q^{t_{2b+1}}$ throughout the formula above and move the pure $m$ terms to the front. In the first summation, reset the index to run from $i=2$ to $i=b+1$ by replacing $i$ with $i-1$ throughout and peel off the new $i=b+1$ term, which is listed on the second line below. 
$$
\begin{array}{c}
\ds m^{s_1+\cdots+s_{2b}}(m^{s_{2b+1}}-1)  + \sum_{i=2}^{b} \left( m^{s_1+\cdots+s_{2b+1-i}}q^{t_{2b+3-i}+\cdots+ t_{2b+1}}(m^{s_{2b+2-i}} -1)  \right) \\[5mm]
\ds + m^{s_1+\cdots+s_{b}}q^{t_{b+2}+\cdots+ t_{2b+1}}(m^{s_{b + 1}} -1) \\[4mm]
\ds + m^{s_1+\cdots+s_{b }}q^{t_{ b + 1}+\cdots+t_{2b+1}} + \sum_{i=1}^{ b } \left(  m^{s_1+\cdots+s_{i-1}}q^{t_{i+1}+\cdots+ t_{2b+1}}(q^{{t_i}} -1)  \right),
\end{array}
$$
which gives formula (\ref{charpolyodd}) once the front terms are incorporated into the first summation as the $i=1$ case and the middle line expression is distributed into terms. 

Consider $k=2b$ and suppose, adapting (\ref{charpolyodd}), that $\chi(M^{s_{2b-1}}R^{t_{2b-1}}\cdots M^{s_1}R^{t_1})$ 
$$ 
\begin{array}{c}
\ds = \sum_{i=1}^{b-1} \left( m^{s_1+\cdots+s_{2b-1-i}}q^{t_{2b+1-i}+\cdots+ t_{2b-1}}(m^{s_{2b-i}} -1)  \right) \\[5mm]
\ds + m^{s_1+\cdots+s_{b}}q^{t_{b+1}+\cdots+t_{2b-1}} - m^{s_1+\cdots+s_{b-1}}q^{t_{b+1}+\cdots+t_{2b-1}}  \\[2mm]
\ds  + m^{s_1+\cdots+s_{b-1}}q^{t_{b}+\cdots+t_{2b-1}}+ \sum_{i=1}^{b-1} \left(  m^{s_1+\cdots+s_{i-1}}q^{t_{i+1}+\cdots+ t_{2b-1}}(q^{{t_i}} -1)  \right)  
\end{array}  
$$
Then $\chi(M^{s_{2b}}R^{t_{2b}}\cdots M^{s_1}R^{t_1})$ is given by
$$
\begin{array}{c}
\ds = q^{2b}\sum_{i=1}^{b-1} \left( m^{s_1+\cdots+s_{2b-1-i}}q^{t_{2b+1-i}+\cdots+ t_{2b-1}}(m^{s_{2b-i}} -1)  \right) \\[5mm]
\ds + q^{2b}m^{s_1+\cdots+s_{b}}q^{t_{b+1}+\cdots+t_{2b-1}} - q^{2b}m^{s_1+\cdots+s_{b-1}}q^{t_{b+1}+\cdots+t_{2b-1}}  \\[2mm]
\ds  + q^{2b}m^{s_1+\cdots+s_{b-1}}q^{t_{b}+\cdots+t_{2b-1}}+ q^{2b}\sum_{i=1}^{b-1} \left(  m^{s_1+\cdots+s_{i-1}}q^{t_{i+1}+\cdots+ t_{2b-1}}(q^{{t_i}} -1)  \right) \\[5mm]
\ds + m^{s_1+\cdots+s_{2b}} - m^{s_1+\cdots+s_{2b-1}}  
\end{array} 
$$
Distribute $q^{t_{2b}}$ throughout and move the pure $m$ terms to the front. In the first summation replace $i$ with $i-1$ throughout to change the range to $i=2$ to $i=b$.  
$$
\begin{array}{c}
\ds  m^{s_1+\cdots+s_{2b}} - m^{s_1+\cdots+s_{2b-1}}  +  \sum_{i=2}^{b} \left( m^{s_1+\cdots+s_{2b-i}}q^{t_{2b+2-i}+\cdots+ t_{2b}}(m^{s_{2b+1-i}} -1)  \right) \\[5mm]
\ds + m^{s_1+\cdots+s_{b}}q^{t_{b+1}+\cdots+t_{2b}} - m^{s_1+\cdots+s_{b-1}}q^{t_{b+1}+\cdots+t_{2b}}  \\[2mm]
\ds  + m^{s_1+\cdots+s_{b-1}}q^{t_{b}+\cdots+t_{2b}}+ \sum_{i=1}^{b-1} \left(  m^{s_1+\cdots+s_{i-1}}q^{t_{i+1}+\cdots+ t_{2b}}(q^{{t_i}} -1)  \right),
\end{array} 
$$
which gives formula (\ref{charpolyeven}) once the front terms are incorporated into the first summation as the $i=1$ case, and the second monomial in the middle and the monomial in the third line enter the second summation as the $i=b$ case. 
\end{proof}

We illustrate the previous theorem by showing some examples where we can immediately write down the characteristic polynomial for a generating sequence based on the exponents of the $M$ and $R$ operators. 

\begin{example} Consider $M^{s_3}R^{t_3}M^{s_2}R^{t_2}M^{s_1}R^{t_1}$. Then its characteristic polynomial is given by the odd case (\ref{charpolyodd}) with $b=1$. 
$$
m^{s_1+s_2}(m^{s_3}-1) + m^{s_1+s_2}q^{t_3} - m^{s_1}q^{t_3} + m^{s_1}q^{t_2+t_3} + q^{t_2+t_3}(q^{t_1}-1)
$$
\end{example}

\begin{example} Consider $M^{s_4}R^{t_4}M^{s_3}R^{t_3}M^{s_2}R^{t_2}M^{s_1}R^{t_1}$. Then its characteristic polynomial is given by the even case (\ref{charpolyeven}) with $b=2$. 
$$
\begin{array}{c}
m^{s_1+s_2+s_3}(m^{s_4}-1) + m^{s_1+s_2}q^{t_4}(m^{s_3}-1) \\[2mm]
  + m^{s_1+s_2}q^{t_3+t_4} \\[2mm]
  + m^{s_1}q^{t_3+t_4}(q^{t_2}-1) + q^{t_2+t_3+t_4}(q^{t_1}-1)
\end{array}
$$
\end{example}

Looking at these examples along with the base case lemmas before the theorem, formulas (\ref{charpolyodd}) and (\ref{charpolyeven}) work a little like the binomial theorem for positive integers where powers on $m$ are decreasing while powers on $q$ are increasing. Additionally, the parenthetical parts of the formulas flip from using a power of $m$ to a power of $q$ halfway through while there are monomials in $m$ and $q$ in the middle.

Motivated by these example cases and the formulas of Theorem \ref{polynthm}, we can visualize all possible conjugate $(m,q)$-ary partitions based on a central region that is either overlapping rectangles of dimensions $m^{s_1+\cdots+s_{b+1}}\times q^{t_{b + 2}+\cdots+t_{2b+1}}$ and $m^{s_1+\cdots+s_{b}}\times q^{t_{b + 1}+\cdots+t_{2b+1}}$ with overlap of dimensions $m^{s_1+\cdots+s_{b}}\times q^{t_{b + 2}+\cdots+t_{2b+1}}$ (in the odd case) or a single rectangle with dimensions $m^{s_1+\cdots+s_{b }}\times q^{t_{ b + 1}+\cdots+t_{2b}}$ (in the even case). Notice how these regions have areas expressed by the middle terms of (\ref{charpolyodd}) and (\ref{charpolyeven}). From those central regions are rectangles branching horizontally to the right and vertically down with dimensions of the forms $m^{\alpha}(m^\beta -1)\times q^\gamma$ and $m^\alpha \times q^\beta(q^\gamma-1)$, respectively. These rectangles have areas given by the terms in the summation parts of the formulas. They are represented visually below (although not to scale) with the $m$-axis going horizontally and the $q$-axis vertically. The central regions are in violet, the first summation terms are in blue, and the second summation terms are in red. For the odd case central region, the overlapped part of the central rectangles is boxed off with dashed lines so that its area matches the negative monomial in the center of (\ref{charpolyodd}), which is subtracted to remove the double count caused by the overlap.
\begin{figure}[h]
\centering
\includegraphics[scale=0.65]{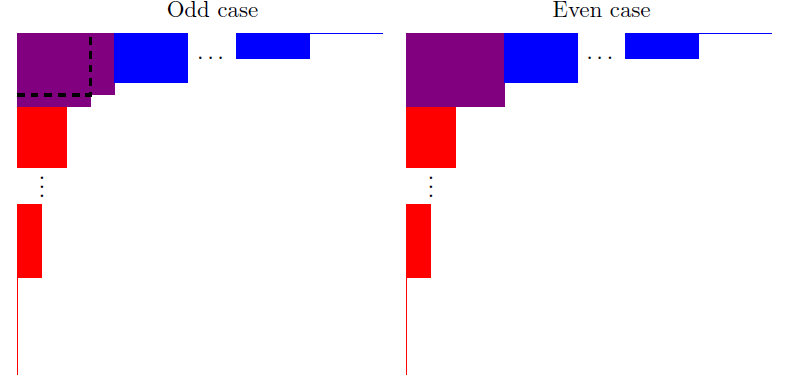}
\caption{Schematic illustration of all conjugate $(m,q)$-ary partition shapes}
\end{figure}
If you replace each rectangular region with a lattice of dots, then these diagrams are just the Ferrers diagrams for the partition rearranged to calculate the number in $n$ in terms of blocks instead of by rows. 

This next  example will prove useful later when studying CMPs in Section 5.

\begin{example} \label{MRk} The sequence $(MR)^k = MRMR \cdots MR$ has all exponents of 1 and $k$ total pairs of blocks for any positive integer $k$. When $k = 2b+1$ is odd, (\ref{charpolyodd}) gives the characteristic polynomial
$$ 
\sum_{i=1}^{b} \left( m^{2b+1-i}q^{i-1}(m -1)  \right) + m^{b+1}q^{b} - m^{b}q^{b} 
 + m^{b}q^{b+1}+ \sum_{i=1}^{b} \left(  m^{i-1}q^{2b+1-i}(q -1)  \right)  
$$
When $k=2b$ is even, (\ref{charpolyeven}) gives 
$$
 \sum_{i=1}^{ b } \left( m^{{2b-i}}q^{i-1}(m -1)  \right) 
+ m^{{b }}q^{b} 
+ \sum_{i=1}^{ b } \left(  m^{i-1}q^{2b-i}(q -1)  \right)  
$$
If we set $m=q$ (as we will in Section 5), then the two expressions simplify significantly to the common formula
$$
\chi((MR)^k) = km^{k-1}(m-1) + m^k =  km^k +(m-k)m^{k-1}.
$$
\end{example}

Using the last theorem, we can classify which integers $n$ have conjugate $(m,q)$-ary partitions for given values of $m$ and $q$.  

\begin{corollary} \label{nclassify}
Let $n\geq 2$ be an integer, and let $m,q \geq 2$. Then $n$ has a conjugate $(m,q)$-ary partition if and only if $n$ equals $\chi(M^{s_k}R^{t_k}\cdots M^{s_1}R^{t_1})$ evaluated at $m$ and $q$ for some integers $s_1, t_1, \ldots, s_k, t_k$. 
\end{corollary}

In later sections, we will make some attempts to determine more concretely which $n$ can and cannot have a conjugate $(m,q)$-ary partition for specific values of $m$ and $q$. In very limited cases we can even count how many such partitions are possible. In general, both of these problems appear to be very difficult to solve.

We can, however, use the theorem to place general bounds on $n$ based on the number of operators used to generate $\lambda$.

\begin{corollary}
Let $m,q\geq 2$. Suppose that $\lambda$ is a conjugate $(m,q)$-ary partition generated by a sequence with a total of $t$ operations of $R$ and $s$ operations of $M$, in some order. Then
\begin{nlist}
\item the minimum $n$ that $\lambda$ can partition is $m^s+q^t-1$ with $\lambda = M^sR^t(1)_m = (1,0\ldots,0,q^t-1)_m$ if $s>0$ or $\lambda = (q^t)_m$ if $s=0$
\item the maximum $n$ that $\lambda$ can partition is $m^sq^t$ with \\ $\lambda = R^tM^s(1)_m = (q^t, 0,\ldots, 0)_m$.
\end{nlist}
\end{corollary}
\begin{proof}
Based on Definition \ref{opdef}, the $t$ operations of $R$ will result in $t$ multiplications by $q$. The $s$ operations of $M$ will result in adding terms of the form $m^{i+1}-m^{i}$, resulting in a maximum value of $m^s-m^{s-1}$. In order to minimize the final value of $n$, we need to avoid multiplying these increasingly large differences of powers of $m$ by powers of $q$. Thus, the minimum occurs when all multiplications by $q$ occur first followed by the additions of the $m$-based terms. In other words, the minimum is $q^t + (m-1) + (m^2-m) + \cdots + (m^s-m^{s-1}) = m^s + q^t-1$ as claimed. This formula is the characteristic polynomial uniquely associated to $M^sR^t$ by Theorem \ref{polynthm} and Lemma \ref{oddbase}. Similarly, the maximum occurs when all the additions occur first and are then scaled up by powers of $q$, i.e., $q^t(1 +(m-1) +\cdots + (m^s-m^{s-1})) = m^sq^t$, which is the characteristic polynomial of $R^tM^s$ by Theorem \ref{polynthm} and Lemma \ref{evenbase}.
\end{proof}

We can also put bounds on $n$ based only on the total number of operators used to generate a partition.

\begin{corollary}
Let $m,q\geq 2$. Suppose that $\lambda$ is a conjugate $(m,q)$-ary partition generated by a sequence of $p$ total operations of $R$ and $M$ in some order. Then 
\begin{nlist}
\item the minimum $n$ that $\lambda$ can partition is min$\{m^s+q^t-1\,|\, s+t=p\}$. That minimum occurs when $s$ is the closest integer to $\frac{p\ln q}{\ln(mq)}$ and $t = p-s$. Or, $n \geq 2m^{\frac{p\ln q}{\ln(mq)}}-1$.
\item the maximum $n$ that $\lambda$ can partition is $($max$\{m,q\})^p$
\end{nlist}
\end{corollary}
\begin{proof}
For any pair $(s,t)$ such that $s+t=p$, we have the minimum and maximum formulas from the last corollary. The formula $m^s+q^t-1$ is minimized when $m^s=q^t$. Setting $t=p-s$ and then solving yields $s = \frac{p\ln q}{\ln(mq)}$. As $s$ must be an integer, we may need to round to the closest integer. We achieve the maximum value from $m^sq^t$ when either $p=s$ or $p=t$ depending on whether $m$ or $q$ is the larger value.
\end{proof}

Finally, we can reverse the direction of our thinking and give bounds on the number of operators that can generate a conjugate $(m,q)$-ary partition for a given value of $n$.

\begin{corollary}
Let $n$ be an integer. Suppose that $\lambda$ is a conjugate $(m,q)$-ary partition of $n$ such that $\lambda$ is generated by a sequence of $p$ total $R$ and $M$ operators in some order. If $max$ denotes the maximum of $m$ and $q$, then 
$$
\left\lceil \frac{\ln(n)}{\ln(max)} \right\rceil \leq p \leq \left\lfloor \frac{\ln\left( \frac{n+1}{2} \right)\ln(mq)}{\ln(m)\ln(q)} \right\rfloor
$$
Equality on the left will occur when $n=max^p$ and $\lambda = M^p$ (if $max = m$) or $R^p$ (if $max = q$). 
\end{corollary}
\begin{proof}
From the previous corollary, $2m^{\frac{p\ln q}{\ln(mq)}}-1 \leq n \leq max^p$.  Solving the inequalities for $p$ and recognizing that $p$ must be an integer leads to the results above. 
\end{proof}

\section{Using modular arithmetic to classify possible conjugate $(m,q)$-ary partitions}

In this section, we ask questions about which values of $n$ have conjugate $(m,q)$-ary partitions and how many such partitions exist. These questions appear to be very difficult to answer in general and involve some very subtle modular arithmetic issues related to specific values of $m$ and $q$. As such, this section includes a lot of examples but fewer general results. The lack of general answers should motivate future study of these questions while the examples may hint at the difficulty in doing so.

We start with four general congruence statements that hold for all pairs $(m,q)$ that can be derived from the $M$ and $R$ operators or from the characteristic polynomial.

\begin{lemma} \label{4congr}
Let $m,q\geq 2$. Suppose that $n$ has a conjugate $(m,q)$-ary partition $\lambda$. Let $d = \gcd( m(m-1), q(q-1) )$. Then $n$ satisfies all of the following:
\begin{nlist}
\item $n \equiv mq \pmod d$ or $n = m^s + q^t -1$.  
\item $n \equiv m^s \pmod{q-1}$, where $s$ is the total number of $M$ operators used to generate $\lambda$.
\item $n \equiv q^t \pmod{m-1}$, where $t$ is total number of $R$ operators used to generate $\lambda$.
\item $n \equiv 1 \pmod{\gcd(m-1,q-1)}$.
\end{nlist}
\end{lemma}
\begin{proof} For (1), note that the second case occurs when the partition is generated by $M^sR^t$ for $s,t \geq 0$ and includes the pure powers $m^s$ and $q^t$ generated by $M^s$ and $R^t$, respectively. Otherwise, we may assume that $\lambda$ is generated by $M^{s_k}R^{t_k}\cdots M^{s_1}R^{t_1}$, where $k\geq 2$ and $s_k,t_1\geq 0$ while all other exponents are strictly positive. Then every term of the characteristic polynomials (\ref{charpolyodd}) and (\ref{charpolyeven}) is divisible by either $m(m-1)$ or $q(q-1)$, except the ``middle terms,'' which are powers of $m$ times  powers of $q$. As $d$ divides both $m^2-m$ and $q^2-q$, the characteristic polynomial reduces to $mq \pmod d$. The proofs for (2) and (3) are very similar to each other. For (2), note that $q \equiv 1 \pmod{q-1}$. Setting $q=1$ in (\ref{charpolyodd}) and (\ref{charpolyeven}) leads to telescoping terms in powers of $m$ that reduce to the claimed power of $m$. For (3), set $m$ to 1 instead. Alternatively, one could prove (2) and (3) using the definitions of the $M$ and $R$ operators and induction on the number of operators that generate $\lambda$. Part (4) can be derived from any of (1), (2), or (3). 
\end{proof}

\begin{remark}
This last lemma generalizes several results for CMPs in \cite{FL21}. Note that when $m=q$, part (1) leads to $n\equiv m \pmod{m(m-1)}$ or $n = m^s+m^t-1\equiv 2m -1 \pmod{m(m-1)}$, recovering \cite[Theorems 3.3]{FL21}. Reducing the modulus to $m-1$ or $m$, respectively, recovers \cite[Propositions 3.1 and  3.2]{FL21} stating that $n\equiv 1 \pmod{m-1}$ in all cases and that $n$ is either $0$ or $-1\pmod{m}$. Note that both (2) and (3) also imply that $n\equiv 1 \pmod{m-1}$ when $m=q$. 
\end{remark}

As a result, we obtain our first general restriction on which $n$ can have conjugate $(m,q)$-ary partitions. 

\begin{proposition}
Let $m,q\geq 3$ be odd positive integers. If $n$ has a conjugate $(m,q)$-ary partition, then $n$ is odd too.
\end{proposition}
\begin{proof}
If both $m$ and $q$ are odd, then Lemma \ref{4congr}(4) implies that $n\equiv 1 \pmod 2$.
\end{proof}

Studying the system of three congruences for $n$ listed in Lemma \ref{4congr}(1--3) sheds some light on which $n$ have conjugate $(m,q)$-ary partitions. We note that $d =\gcd( m(m-1), q(q-1) )\geq 2$ and is always even. So, at the worst, checking $n\pmod d$ gives even or odd information about $n$ when $n\neq m^s+q^t-1$. 
Also, all positive exponent powers of $m$ and $q$ are congruent modulo $d$ to $m$ and $q$, respectively. 

\begin{example} Using mainly (1) from Lemma \ref{4congr}, we examine some specific systems to determine restrictions on $n$ such that $n$ has a conjugate $(m,q)$-ary partition.
\begin{nlist}
\item When $(m,q)=(5,3)$, we learn that $n\equiv 1 \pmod 2$ as $d=2$ or that $n=5^t+3^s-1$, which is also always odd.  
\item When $(m,q) = (6,3)$, we see that $d = 6$ so that $n = 6^s+3^t-1$ or $n\equiv 18 \equiv 0 \pmod 6$. The condition $n \equiv 3^t \pmod 5$ additionally implies that $n \not\equiv 0 \pmod 5$. Putting these together we further learn that $n\not\equiv 0\pmod{30}$. 
\item When $(m,q) = (7,3)$, we have $d = 6$ so that $n = 7^s+3^t-1$ or $n\equiv 21 \equiv 3 \pmod 6$. Notice that other than the pure powers of $7$, we have $n\equiv 3\pmod 6$ when $n = 7^s+3^t-1$, $t>0$. 
\item When $(m,q) = (8,3)$, we have $n\equiv 24 \pmod 2$ as $d=2$. Also, $n\equiv 3^t \pmod 7$. Thus, $n =  8^s+3^t-1$, or $n$ is both even and never $0\pmod 7$. 
\item When $(m,q) = (13,7)$, we have $d=6$ so that $n=13^s+7^t -1 \equiv 1 \pmod 6$ or $n\equiv 13\cdot 7 \equiv 1 \pmod 6$. Thus, $n\equiv 1 \pmod 6$ for all $n$. 
\item When $(m,q) = (15,7)$, we have $d = 42$ so that $n\equiv 15\cdot 7\equiv 21\pmod{42}$ or $n = 15^s+7^t-1 \equiv 15+7-1 \equiv 21 \pmod{42}$, as long as $s,t>0$. Therefore, we have $n=15^k$, $n=7^k$, or $n\equiv 21\pmod{42}$. 
\end{nlist}
\end{example}

These examples show how we can eliminate certain $n$ values as never having conjugate $(m,q)$-ary partitions, but these congruences form a fairly coarse filter which can only take us so far in understanding which $n$ can occur. So, we next examine some other general restrictions that we can derive from modular arithmetic conditions. These results demonstrate that in certain cases we can show that the only way to generate conjugate $(m,q)$-ary partitions of powers of $m$ or $q$ is to use trivial rectangles, where there are either $q^k$ copies of 1 or one copy of $m^k$.

\begin{theorem} \label{powersunique}
Let $m,q\geq 2$, and let $d = \gcd( m(m-1), q(q-1) )$. 
\begin{nlist}
\item If $q\not\equiv mq \pmod d$, then the only conjugate $(m,q)$-ary partition of $q^k$ is generated by $R^k$. 
\item If $m\not\equiv mq \pmod d$, then the only conjugate $(m,q)$-ary partition of $m^k$ is generated by $M^k$. 
\end{nlist}
\end{theorem}
\begin{proof} The proofs are essentially identical, so we only show the first case. Note that $R^k$ generates a rectangular partition of $q^k$ for all $k$. Note that $M^s$ generates a partition of $m^s$. If $M^s$ generates a partition of $q^k$,  then $m^s = q^k$ implies that $m\equiv q \pmod d$, which implies that $q\equiv q^2 \equiv mq \pmod d$, a contradiction.
Our hypothesis implies that $m\not\equiv 1 \pmod d$. Thus, $M^sR^t$ with $s,t >0$ generates a partition of $m^s+q^t - 1 \equiv m+q-1  \not\equiv q \equiv q^k \pmod d$, which means that $M^sR^t$ cannot generate a partition of $q^k$. Finally, in all other cases, we partition an $n$ such that  $n \equiv mq \not\equiv q \pmod d$ by Lemma \ref{4congr}(1), but then $n \not\equiv q \equiv q^k \pmod d$. Thus, $R^k$ is the only way to generate such a partition of $q^k$.
\end{proof}
 
When $m$ is even and $q$ is odd, then conjugate $(m,q)$-ary partitions of odd numbers $n$ only occur when $n$ is a power of $q$, and such partitions only arise from trivial cases formed by $q^k$ parts that are all 1s. 

\begin{corollary}
Let $m$ be even and $q$ be odd, both positive integers at least 2. Let $n$ be an odd integer. Then $n$ has a conjugate $(m,q)$-ary partition if and only if $n=q^k$ for some $k\geq 0$. Moreover, the only conjugate $(m,q)$-ary partition of $q^k$ is generated by $R^k$. 
\end{corollary}
\begin{proof}
 Recall that $d = \gcd( m(m-1), q(q-1) )$ is always even. If odd $n$ has a conjugate $(m,q)$-ary partition, then $n = m^s+q^t-1$ or $n\equiv mq \pmod d$. The latter case implies that $n$ is even, so $n = m^s+q^t-1$ is the only possibility. If $s>0$, then $n$ would also be even so that $s=0$ and $n$ is a power of $q$. Finally, an even $m$ and odd $q$ imply that $q\not\equiv mq \pmod d$, which proves the second part of our claim based on the last theorem. 
\end{proof}

While we can obtain a general odd-even result for powers of $q$ as above, the situation is more subtle for other relatively prime pairs $(m,q)$. For example, $R^3$ generates a conjugate $(m,q)$-ary partition $(q^3)_m$ of $q^3$ for any pair $(m,q)$. However, for the pair $(5,3)$, the sequence $M^2R$ generates the non-rectangular partition $(1,0,2)_5$ of $3^3$ whose $(3,5)$-ary conjugate is $(1,24)_3$ generated by $MR^2$. 

\begin{example} We now look back at how these last few results apply to our earlier examples.
\begin{nlist}
\item For $(m,q)=(5,3)$ or $(13,7)$, we learn nothing new as $m,q\equiv mq \pmod d$ in both cases as $d=2$. As indicated above for $(5,3)$, the powers $3^k$ are not uniquely partitioned for $k\geq 3$. It is not clear whether the powers of $5$ are uniquely partitioned or not. Similarly, for $(13,7)$ it is unknown whether the powers of 13 and 7 are uniquely partitioned or not. Calculated data has not found any counterexamples though. 
\item For $(m,q)=(6,3)$, we have an even and odd pairing. Thus, we know that the only odd numbers that have a conjugate $(6,3)$-ary partition are powers of $3$ and that each power of $3$ is uniquely partitioned by the rectangle generated by $R^k$. 
\item For $(m,q) = (7,3)$, we have $d=6$. As $m=7\not\equiv 21 \pmod d$, we see that the only conjugate $(7,3)$-ary partitions of $7^k$ are generated uniquely by $M^k$. 
\item For $(m,q) = (8,3)$, we have essentially the same situation as $(6,3)$ given the even-odd pairing. 
\item For $(m,q) = (15,7)$, we have $d=42$ and note that neither $7$ nor $15$ are congruent to $15\cdot 7 \equiv 21 \pmod{42}$. Thus, powers of $15$ and powers of $7$ only have the conjugate $(15,7)$-ary partitions generated by $M^k$ and $R^k$, respectively. 
\end{nlist}
\end{example}

In order to demonstrate how subtle some of the issues are in these situations, we closely investigate the case of $(m,q) = (4,3)$ and point out some observations. For the rest of this section, we assume that $(m,q) = (4,3)$. 

\begin{example}
Let $(m,q) = (4,3)$. Note that $d = \gcd(4\cdot 3, 3\cdot 2) = 6$. Also note that $4\cdot 3 \equiv 0 \pmod 6$ and that $n = m^s+q^t -1 \equiv m+q-1 = 4+3-1 \equiv 0 \pmod 6$ when $s,t >0$ so that either $n=4^k$, $n=3^k$, or $n\equiv 0 \pmod 6$ if $n$ has a conjugate $(4,3)$-ary partition. The other congruences of Lemma \ref{4congr} do not give extra information. Notice that neither $4$ nor $3$ are congruent to $12\pmod 6$ so that Theorem \ref{powersunique} applies in both cases. \\[2mm]
\textit{Fact:} The powers $4^k$ and $3^k$ are each uniquely partitioned by $M^k$ and $R^k$, respectively.  \vspace{2mm}

We next look closer at the $n\equiv 0 \pmod 6$ cases and discover more information by looking at $n \pmod{24}$. \\[2mm]
\textit{Fact:} Let $\lambda$ partition $n$. Then $n\equiv 12 \pmod {24}$ if and only if $\lambda$ is generated by $R^bM R^{2a}$, where $a,b\geq 0$ but not both $0$. 
\begin{proof}
We analyze partitions $\lambda$ based on how many $M$ operators occur in their generating sequences $A$. We ignore the trivial partition of $1$ generated by the empty sequence $A$. If $A=R^k$, $k \geq 1$ (no $M$ operators), then $n = 3^k \not\equiv 0\pmod 4$ and so cannot be $12\pmod{24}$. Suppose that $A = R^bMR^k$, $b,k\geq 0$ (only one $M$ operator). Then $n = 3^b(4+ 3^k -1) = 3^b(3^k + 3)$. Note that $3^k \pmod{24}$ is $3$ when $k$ is odd and $9$ when $k$ is even and positive. When $k$ is odd, $n\equiv 3^b\cdot 6 \pmod{24}$, which is either $6$ or $18\pmod{24}$. When $k$ is even and positive, $n\equiv 3^b\cdot 12 \equiv 12 \pmod{24}$. Next consider $A = R^c MR^b MR^k$, where $b,c,k\geq 0$ but not all $0$. The second $M$ on the left adds $4^2-4 = 12$ to the previous $n$ value. Thus, it partitions a number that is either $0,6,18 \pmod{24}$, and the $R^c$ then multiplies by $3^c$, which still leaves us partitioning a number that is either $0,6,18 \pmod{24}$. For an induction, suppose that $A$ is a sequence that is either $M^k$ ($k\geq 2$) or contains at least two copies of $M$ that partitions a number $n$ that is either $0,6,18 \pmod{24}$. Then if $A=M^k$, $MA = M^{k+1}$ while $RA$ has at least 2 copies of $M$ and partitions $3\cdot 4^k \equiv 0 \pmod{24}$. In the other case, $MA$ partitions $n+4^{s+1}-4^s = n+ 4^{s-2}\cdot 48\equiv n \pmod{24}$, for some $s\geq 2$.  Also, $RA$ partitions $3\cdot n$, which will be either $0,6,18\pmod{24}$. Therefore, sequences of the form $R^b M R^{2a}$, where $a,b$ are not both $0$, are the only ones that generate a partition of $n$ with $n\equiv 12\pmod{24}$. 
\end{proof}

Examining this case even more closely, the sequences  $R^b M R^{2a}$, where $a,b$ are not both $0$, have $n = 3^b(3^{2a} + 3)$, which will always be $4\pmod 8$. If $a,b\geq 1$, then $n\equiv 0\pmod 9$. If $n = 3^{2a} + 3$, where $a\geq 1$, then $n\equiv 3 \pmod 9$ while if $n = 3^b\cdot 4$, where $b\geq 1$, then $n\equiv 0,3 \pmod 9$. In other words, such $n\not\equiv 6\pmod 9$. By the Chinese Remainder Theorem, $4\pmod 8$ and $6\pmod 9$ uniquely lift to $60\pmod{72}$. As $60\pmod{72}$ descends to $12\pmod{24}$, we have proven the following fact. \\[2mm]
\textit{Fact:} Let $\lambda$ partition $n$. Then $n\not\equiv 60\pmod{72}$. 
\vspace{2mm}

Finally, in the $n \equiv 12\pmod{24}$ case, we can count how many ways one can partition $n$ using conjugate $(4,3)$-ary partitions. \\[2mm]
\textit{Fact:} Suppose that $n\equiv 12 \pmod{24}$. Then 
\begin{nlist}
\item there are exactly 2 partitions of $n$ if $n = 3^b\cdot 4$, where $b\geq 1$
\item there is exactly 1 partition of $n$ if $n= 3^b(3^{2a}+3)$, where $a\geq 2$, $b\geq 0$
\item there is no partition of $n$ otherwise.
\end{nlist}
\begin{proof}
If $n\equiv 12 \pmod{24}$, then $n$ only has a partition if $n = 3^b(3^{2a}+3)$. When $a=0$, we have $n=3^b\cdot 4$ with the partition generated by $R^b M$ for all $b\geq 1$, but that can also be attained when $n = 3^{b-1}(3^2+3) = 3^{b}\cdot 4$, which is generated by $R^{b-1}MR^2$ for $b-1\geq 0$. This first case covers all $a=0$ partitions, and the second case covers all $a=1$ partitions. For $a\geq 2$ and $b\geq 0$, $R^b M R^{2a}$ partitions $n = 3^b(3^{2a}+3)= 3^{b+1}(3^{2a-1}+1)$, where $3^{2a-1}+1> 4$. By the Fundamental Theorem of Arithmetic, none of these can equal one of the earlier cases nor can they equal each other with different exponents. All cases are thus covered by one of the three claimed cases.
\end{proof}
\end{example}

In this last example, we used the results of this section to discover some values of $n$ that can be represented by a $(4,3)$-ary partition. Computationally, we have seen that there are many additional values of $n$ beyond those covered in the example such that $n \equiv 0\pmod 6$ and the number of conjugate $(4,3)$-ary partitions is positive. For example, considering the $8000$ possible values of $n$ which are 0 modulo 6 and $n\leq 48000$, we found that $3784$ of them have at least 1 conjugate $(4,3)$-ary partition. However, only $17$ of these partitions are for $n$ values in the $12$ modulo $24$ congruence class. Combining the $0$, $6$, and $18$ congruences classes, $3767$ of the $6000$ eligible $n$ values can be expressed with a $(4,3)$-ary partition. We have not found a way to predict or further classify which specific $n$ values will appear.

In general, the results in the $(4,3)$-ary example give evidence to the difficulty of identifying all values of $n$ that have $(m,q)$-ary partitions. We can use facts in this section as in the last example to narrow down possible $n$ values in certain eligible congruence classes, but this does leave unanswered questions.

\begin{ques}
Are there any general theorems that tie together the observations in the examples above into a more coherent story? If specific $n$ cannot be classified, can one at least calculate the density of integers $n$ that have a conjugate $(m,q)$-ary partition? 
\end{ques}

\section{Self-conjugate CMPs of powers of \textit{m}}

In this section we consider conjugate $(m,m)$-ary partitions, which are the CMPs of \cite{FL21}. More specifically, we study those that are self-conjugates. In Corollary \ref{selfconj} we saw that self-conjugate CMPs are generated by $M,R$ sequences whose powers are a palindrome. Even this group is too large to study fully right now so we further restrict to self-conjugate CMPs that partition some power of $m$. In \cite[Lemma 4.2]{FL21}, it was noted (in our new language) that $(MR)^m$ is a self-conjugate CMP of $m^{m+1}$. Corollary \ref{selfconj} proves this partition is self-conjugate and Example \ref{MRk} indicates it is a partition of $m\cdot m^{m} + (m-m)m^{m-1} = m^{m+1}$. We also note that this family consists of non-square partitions as all square self-conjugate CMPs are of the form $R^kM^k$ and partition $m^{2k}$ for all $k$.

Inspired by these families of self-conjugate CMPs of a power of $m$, we used computational experiments to discover other infinite families of such non-square partitions. In the following results, we will rely heavily on Remark \ref{usefulremark} and Lemma \ref{commlemma} and make use of a mixture of all four operators $M,R,S$, and $J$. We start with a lemma that will aid our calculations. This lemma is true for general pairs $(m,q)$. 

\begin{lemma} \label{calclemma} Let $m,q\geq 2$. 
Let $A$ denote a sequence of $M$ and $R$ operators with $s$ copies of $M$ and $t$ copies of $R$ in any order. Then
\begin{nlist}
\item $\chi(MJA) = (m^{s+1}-m^s) + (q^{t+1}-q^t) + \chi(A)$, where $MJA$ has $s+1$ copies of $M$ and $t+1$ copies of $R$. 
\item $\chi(RSA) = mq\chi(A)$, where $RSA$ has $s+1$ copies of $M$ and $t+1$ copies of $R$. 
\item $\chi(MJRSA) = (m^{s+2}-m^{s+1}) + (q^{t+2}-q^{t+1}) +mq \chi(A)$, where $MJRSA$ has $s+2$ copies of $M$ and $t+2$ copies of $R$. 
\end{nlist}
\end{lemma}
\begin{proof}
By combining Definition \ref{opdef} and Corollary \ref{opcount}, we can see the first claim in (1). By Remark \ref{usefulremark} and Lemma \ref{commlemma}, we see that $MJA = MAJ = MAR$, and so we have added an additional $M$ and $R$ to the sequence. The first part of (2) follows from Definition \ref{opdef} while $RSA = RAS  = RAM$ gives the second part. Combining (1) and (2) yields (3).
\end{proof}

We now present our new infinite families of self-conjugate CMPs. 

\begin{theorem}
Let $m=q\geq 2$ be integers. Set $B = m^{k+1}-m^k - m^{k-1} - \cdots - m$. Then $M^k(MR)^B R^k$ generates a self-conjugate CMP of $m^{B+k+1}$ for all integers $k\geq 0$. 
\end{theorem}
\begin{proof} The self-conjugate claim is verified by Corollary \ref{selfconj}. 
By Remark \ref{usefulremark} and Lemma \ref{commlemma}, $M^k(MR)^B R^k = M^k(MR)^B J^k =  M^k J^k (MR)^B = (MJ)^k (MR)^B$.
By $k$ applications of Lemma \ref{calclemma}(1), we have the polynomial  $q^{B+k}-q^{B} + m^{B+k}-m^B  + \chi((MR)^B) = 2(m^{B+k}-m^B)+ \chi((MR)^B)$. Example \ref{MRk} gives us $\chi((MR)^B)$, and so we see that $M^k(MR)^B R^k$ generates a self-conjugate CMP of 
$$
\begin{array}{c}
2(m^{B+k}-m^B) + Bm^B + (m-B)m^{B-1} \\[2mm]
= 2(m^{B+k}-m^B) + (m^{B+k+1}-m^{B+k} - m^{B+k-1} - \cdots - m^{B+1}) \\
+ (m^B - (m^{B+k}-m^{B+k-1} - m^{B+k-2} - \cdots - m^{B}) ) \\[2mm]
= 2(m^{B+k}-m^B) + (m^{B+k+1}-m^{B+k} - m^{B+k-1} - \cdots - m^{B+1}) \\
+ (m^B - m^{B+k}+m^{B+k-1} + m^{B+k-2} + \cdots + m^{B} ) \\[2mm]
= 2(m^{B+k}-m^B) + m^{B+k+1} -2m^{B+k} + 2m^B = m^{B+k+1}
\end{array}
$$
\end{proof}

\begin{theorem}
Let $m=q\geq 2$ be integers.  Set $B = m^2-(2k+1)m$. Then $(MR)^kM(MR)^BR(MR)^k$ generates a self-conjugate CMP of $m^{B+2(k+1)}$ for all integers $0 \leq k \leq \frac{m-1}{2}$. 
\end{theorem}
\begin{proof} The self-conjugate claim is verified by Corollary \ref{selfconj} again. Note that the upper bound on $k$ is necessary to ensure that $B$ is non-negative. By Remark \ref{usefulremark} and Lemma \ref{commlemma}, $(MR)^k M(MR)^B R(MR)^k = (MJRS)^k MJ (MR)^B$. Example \ref{MRk} gives
$$
\begin{array}{c}
\chi((MR)^B) = Bm^B +(m-B)m^{B-1} \\[2mm]
= m^{B+2} - (2k+1)m^{B+1} +m^B - m^{B+1} +(2k+1)m^B \\[2mm]
= m^{B+2} - 2(k+1)m^{B+1} + 2(k+1)m^B
\end{array}
$$
Lemma \ref{calclemma}(1) gives
$$
\begin{array}{c}
\chi(MJ(MR)^B) =  2(m^{B+1}-m^B) + m^{B+2} - 2(k+1)m^{B+1} + 2(k+1)m^B \\[2mm]
=  m^{B+2} - 2k(m^{B+1} -m^B),
\end{array}
$$
where the sequence is generated by $B+1$ copies each of $M$ and $R$. Then Lemma \ref{calclemma}(3) gives
$$
\begin{array}{c}
\chi((MJRS) MJ(MR)^B) = 2(m^{B+3}-m^{B+2}) + m^2(m^{B+2} - 2k(m^{B+1} -m^B)) \\[2mm]
= 2(m^{B+3}-m^{B+2}) + (m^{B+4} - 2k(m^{B+3} -m^{B+2})) \\[2mm]
= m^{B+4} - 2(k-1)(m^{B+3} -m^{B+2}),
\end{array}
$$
where the sequence is now generated by $B+3$ copies each of $M$ and $R$. Repeating this process will progressively increase powers by 2 while removing another pair of the lower degree terms for each application of $(MJRS)$. In the end,
$$
\chi((MJRS)^k MJ (MR)^B) = m^{B+2 +2k} = m^{B+2(k+1)}.
$$
\end{proof}

When $m$ is even, we can also construct several intertwined families of self-conjugate CMPs distinct from those above.  We will build these families in stages.

\begin{lemma}
Let $m=q\geq 2$ be integers, and let $k\geq 1$ be an integer. Then $(MR^2MRM^2R)^{k-1}$ consists of $4k-4$ operations each of $M$ and $R$ and generates a self-conjugate CMP of $2(k-1)m^{4k-3}+m^{4k-4}-2(k-1)m^{4k-5}$. 
\end{lemma}
\begin{proof}
The self-conjugate claim follows from the symmetry in the exponents used, while the operator count is clear from the formula. We then proceed by induction on $k$. The case $k=1$ is trivial, and the case $k=2$ follows because
$$
\begin{array}{c}
\chi(MR^2MR M^2 R) = \chi(MJ(RS)^2(MR)) = 2(m^4-m^3) + m^4(2m-1) \\[2mm]
 = 2m^5 +m^4 -2m^3. 
\end{array}
$$
Suppose that the result holds for some $k\geq 2$. We then investigate $(MR^2MRM^2R)^{k}$ by applying the sequence $MR^2MRM^2R$ in three segments to the formula for the inductive hypothesis. Starting with the application $M^2R$ we use Definition \ref{opdef} and Corollary \ref{opcount} to get 
$$
\begin{array}{c}
m(2(k-1)m^{4k-3}+m^{4k-4}-2(k-1)m^{4k-5}) + m^{4k-2}-m^{4k-4} \\[2mm]
= (2k-1)m^{4k-2}+m^{4k-3}-(2k-1)m^{4k-4} 
\end{array}
$$
because $R$ multiplies the polynomial by $m$ while increasing only the number of $R$ operators, and $M^2$ adds $m^{4k-2}-m^{4k-4}$ because the inductive hypothesis formula had $4k-4$ operations of $M$. We now have $4k-3$ operations of $R$ and $4k-2$ of $M$. We then apply $MR$ to get
$$
\begin{array}{c}
m((2k-1)m^{4k-2}+m^{4k-3}-(2k-1)m^{4k-4}) + m^{4k-1}-m^{4k-2}   \\[2mm]
= 2km^{4k-1}-(2k-1)m^{4k-3}
\end{array}
$$
with operation counts of $4k-2$ for $R$ and $4k-1$ for $M$. Finally, apply $MR^2$ to get the necessary result
$$
\begin{array}{c}
m^2(2km^{4k-1}-(2k-1)m^{4k-3}) + m^{4k}-m^{4k-1} \\[2mm]
= 2km^{4k+1} + m^{4k} - 2km^{4k-1}
\end{array}
$$
\end{proof}

Setting $k=\frac{m}{2}+1$ for even values of $m=q$ above yields a family of self-conjugate CMPs. 

\begin{corollary}
If $m=q\geq 2$ are even integers, then $(MR^2MR M^2 R)^{\frac{m}{2}}$ generates a self-conjugate CMP of $m^{2(m+1)}$.
\end{corollary}

The next step in building our new family of self-conjugate CMPs is to cap the sequences above with $M^2R$ on the left and $MR^2$ on the right to set the base cases for the family members.

\begin{lemma}
Let $m=q\geq 2$ be any integers, and let $k\geq 1$ be an integer. Then $M^2R(MR^2MRM^2R)^{k-1}MR^2$ consists of $4k-1$ operations each of $M$ and $R$ and generates a self-conjugate CMP of $2km^{4k-1}+m^{4k-2}-2km^{4k-3}$. 
\end{lemma}
\begin{proof}
The symmetry of exponents again gives the self-conjugacy while the previous lemma gives the operation count as we are adding three more operations each of $M$ and $R$. Using Remark \ref{usefulremark} and Lemma \ref{commlemma}, rewrite  $M^2R(MR^2MRM^2R)^{k-1}MR^2$ as  $(MJ)^2RS(MR^2MRM^2R)^{k-1}$ and get the polynomial
$$
\begin{array}{c}
2(m^{4k-1}-m^{4k-3}) + m^2(2(k-1)m^{4k-3}+m^{4k-4}-2(k-1)m^{4k-5}) \\[2mm]
= 2km^{4k-1}+m^{4k-2}-2km^{4k-3}
\end{array}
$$
using Lemma \ref{calclemma} and the fact that applying $RS$ results in $4k-3$ operations.
\end{proof}

We now have another family of self-conjugate CMPs for even values of $m=q$ when we set $k = \frac{m}{2}$.

\begin{corollary}
Let $m=q\geq 2$ be even integers. Then the sequence of operators $M^2R(MR^2MRM^2R)^{\frac{m-2}{2}}MR^2$ generates a self-conjugate CMP of $m^{2m}$. 
\end{corollary}

Finally, we can apply repeated operations of $M^2R^2$ to the left and right to produce more self-conjugate CMPs for even values of $m$.

\begin{lemma}
Let $m=q\geq 2$, $\ell\geq 0$, and $k\geq 1$ be integers. Then the sequence $(M^2R^2)^\ell M^2R(MR^2MRM^2R)^{k-1}MR^2 (M^2R^2)^\ell$ consists of $4(k+\ell)-1$ operations each of $M$ and $R$ and generates a self-conjugate CMP of 
$$
2(k+\ell)m^{4(k+\ell)-1}+m^{4(k+\ell)-2}-2(k+\ell)m^{4(k+\ell)-3}.
$$ 
\end{lemma}
\begin{proof}
Again the claims on self-conjugacy and operation counts are clear based on preceding results. We proceed by induction on $\ell$ with our last lemma serving as the base case $\ell = 0$. Suppose the result holds for some $\ell\geq 0$. Then we investigate what happens when we cap the sequence on both the left and right with an extra copy of $M^2R^2$. By Remark \ref{usefulremark} and Lemma \ref{commlemma}, that action is the same as applying $(MJ)^2(RS)^2$ on the left. Using Lemma \ref{calclemma}, the $(RS)^2$ operator will multiply the polynomial by $m^4$ while increasing the number of operations to $4(k+\ell)+1$ each. Then $(MJ)^2$ will add $2(m^{4(k+\ell)+3} - m^{4(k+\ell)+1})$, resulting in
$$
2(k+\ell+1)m^{4(k+\ell)+3}+m^{4(k+\ell)+2}-2(k+\ell+1)m^{4(k+\ell)+1},
$$
which meets the needs for the induction.
\end{proof}

We first note the special case when $k=1$ and $\ell = \frac{m-2}{2}$ as it applies for all $m=q\geq 2$. 

\begin{corollary}
Let $m=q\geq 2$ be even integers. Then the sequence of operators 
$(M^2 R^2)^{\frac{m-2}{2}} M^2R MR^2 (M^2 R^2)^{\frac{m-2}{2}}$ 
generates a self-conjugate CMP of $m^{2m}$. 
\end{corollary}

For general $k$ and $\ell$ we get a more robust family that includes both of those shown in the last two corollaries. We also substitute $k$ for $k-1$ in the formula to simplify it a bit.  The count in the corollary follows by counting the non-negative pairs $(k,\ell)$ that sum to $\frac{m-2}{2}$.

\begin{corollary}
Let $m=q\geq 2$ be even integers, and let $k,\ell \geq 0$ be integers such that $k+\ell = \frac{m-2}{2}$. Then $(M^2R^2)^\ell M^2R(MR^2MRM^2R)^{k}MR^2 (M^2R^2)^\ell$ generates a self-conjugate CMP of $m^{2m}$. For a given even value of $m=q$, there are $\frac{m}{2}$ distinct self-conjugate CMPs of $m^{2m}$ in this family. 
\end{corollary}

\begin{disc} \label{summarylist}
We now have three families of non-square, self-conjugate CMPs of powers of $m$ for arbitrary values of $m$ and two more families for even $m$.
\begin{nlist} 
\item $(MR)^m$ generates a  self-conjugate CMP of $m^{m+1}$ for all integers $m\geq 2$.
\item $M^k(MR)^{m^{k+1}-m^k - m^{k-1} - \cdots - m} R^k$ generates a  self-conjugate CMP of \\$m^{m^{k+1}-m^k - m^{k-1} - \cdots - m +k+1}$ for any integers $m\geq 2$ and  $k\geq 0$.
\item $(MR)^kM(MR)^{m^2-(2k+1)m} R(MR)^k$ generates a  self-conjugate CMP of \\ $m^{m^2-(2k+1)m+2(k+1)}$  for any integers $m\geq 2$ and  $0 \leq k \leq \frac{m-1}{2}$.
\item $(MR^2MR M^2 R)^{\frac{m}{2}}$ generates a  self-conjugate CMP of $m^{2(m+1)}$ for even integers $m\geq 2$.
\item $(M^2R^2)^\ell M^2R(MR^2MRM^2R)^{k}MR^2 (M^2R^2)^\ell$ generates a  self-conjugate CMP of $m^{2m}$  for even integers $m\geq 2$ and $k,\ell\geq 0$ such that $k+\ell = \frac{m-2}{2}$.
\end{nlist}
The square self-conjugate CMPs are generated by $R^kM^k$ and partition $m^{2k}$. Similarly, any family above can be capped on the left with $R^k$ and the right with $M^k$ to generate a self-conjugate, non-square CMP of a power of $m$ that is $m^{2k}$ times larger than the one given by the family above. 
Because of the squares, every even power of \textit{any} $m$ possesses a self-conjugate CMP. 

Finally, consider even $m$ only. Family (1) partitions odd powers of $m$; Family (2) can generate either even or odd powers depending on the choice of $k$; and Families (3), (4), and (5) partition even powers of $m$. Furthermore, by applying $R^k$ on the left and $M^k$ on the right to Family (1), we can generate self-conjugate CMPs for all odd powers of $m$ as long as that odd power is at least $m+1$.   
\end{disc}

\begin{ques}
When $m$ is even, is $m^{m+1}$ the smallest odd power of $m$ that posseses a self-conjugate CMP? 
\end{ques}

This question is related to the following question.

\begin{ques}
For any $m$, does $(MR)^m$ always generate the smallest non-square CMP of a power of $m$?
\end{ques}

When $m$ is odd, note that Families (1) and (3) clearly partition even powers of $m$. The same is true for Family (2): Notice that $B=m^{k+1}-m^k - m^{k-1} - \cdots - m$ has $k+1$ terms. When $k$ is even, $B$ has an odd number of odd terms and so is odd. Thus $B+k+1$ is even. When $k$ is odd, $B$ has an even number of odd terms and is even. Thus $B+k+1$ is even again. We also already noted that squares always partition even powers of $m$. In other words, we do not yet have any examples of self-conjugate CMPs for odd powers of odd $m$. 

\begin{ques}
When $m$ is odd, are there any self-conjugate CMPs of odd powers of $m$?
\end{ques}

\section{Adjusting the bases $(m,q)$ using powers}

In this section we look at how different bases can relate different partitions to each other. We first consider bases $(m,q) = (w^a,w^b)$, for some integer $w\geq 2$, and how they relate to $(w,w)$. We also consider the bases $(m^a,q^b)$ and their relationship to $(m,q)$.  

Let $(m,q) = (w^a,w^b)$, where $a, b \geq 1$ for some integer $w\geq 2$. We can now generalize \cite[Theorem 3.4]{FL21}, which counts the number of CMPs for $n$ of the form $m^s+m^t-1$. 

\begin{theorem} \label{commonpowerthm}
Let $(m,q) = (w^a,w^b)$, where $a, b \geq 1$ for some integer $w\geq 2$. Let $n\equiv -1 \pmod{w}$. Then 
\begin{nlist}
\item there are exactly 2 partitions of $n$ if $n=w^{as}+w^{bt}-1$, where $as \equiv 0 \pmod b$, $bt \equiv 0 \pmod a$, and $as\neq bt$,
\item there is exactly 1 partition of $n$ if $n=w^{as}+w^{bt}-1$, where at least one of the three conditions listed above does not hold,
\item there is no partition of $n$ otherwise. 
\end{nlist}
In the first case, the two partitions are generated by the distinct sequences $M^sR^t$ and $M^{bt/a}R^{as/b}$. In the second case, $M^sR^t$ is the unique sequence that partitions the given $n$. 
\end{theorem}
\begin{proof}
Lemma \ref{4congr}(1) shows that if $n$ has a conjugate $(m,q)$-ary partition, then $n=m^s+q^t-1$ or $n\equiv mq \pmod d$, where $d = \gcd( m(m-1),
 q(q-1) )$. So, $n\equiv 0 \pmod w$ in the latter case given that $w$ divides $d,m,q$, meaning that we are only concerned with $n=m^s+q^t-1$. In the case that $s$ or $t$ is 0 we note that this gives either a partition of a power of $m$, generated by $M^s$, or a power of $q$, generated by $R^t$. In either case $n\equiv 0\pmod w$, so $s$ and $t$ must both be nonzero. The partitions generated by $M^sR^t$ where both $s,t>0$ will have $n = m^s+q^t-1 \not\equiv 0 \pmod w$, so they are the only ones that can
   possibly partition an $n\equiv -1 \pmod w$. Note that $M^sR^t$ generates a partition of  $m^s+q^t-1 = w^{as}+w^{bt}-1$. Suppose that $w^{as}
   +w^{bt}-1 = w^{as_2}+w^{bt_2}-1$. If $t= t_2$, then $s = s_2$ as well and both come from the same partition generated by $M^{s}R^{t}$. 
   So, we may assume without loss of generality that $bt$ is the smallest of the four exponents and that $t < t_2$. Thus, 
   $w^{as-bt}+1 = w^{as_2-bt}+w^{bt_2-bt}$, where all exponents are non-negative. As $bt_2-bt>0$, it must be that $as_2=bt$ and $bt_2=as$. 
   Note that $as=bt$ if and only if $s=s_2$ and $t=t_2$. Given a pair $(s,t)$, we have $M^sR^t$ generates a partition of the same number $n$ as
    the distinct partition generated by $M^{s_2}R^{t_2}$ if and only if $bt \equiv 0 \pmod a$, $as \equiv 0 \pmod b$, and $as\neq bt$. In such 
    a case, $s_2 = bt/a$ and $t_2 = as/b$, both integers.
\end{proof}

We now change focus to look at conjugate $(m^a,q^b)$-ary partitions. We derive an injective map from these partitions into the set of conjugate $(m,q)$-ary partitions.

\begin{proposition} Let $m,q\geq 2$, and $a,b$ be positive integers. There is a one-to-one correspondence between the set of all conjugate $(m^a,q^b)$-ary partitions and conjugate $(m,q)$-ary partitions of the form $M^{as_k}R^{bt_k}\cdots M^{as_1}R^{bt_1}(1)_m$.
\end{proposition}
\begin{proof}
Consider a function taking the conjugate $(m^a,q^b)$-ary partition generated by 
 $M^{s_k}R^{t_k}\cdots M^{s_1}R^{t_1}$ (with cumbersome subscripts of $m^aq^b$ omitted) to the conjugate $(m,q)$-ary partition generated by
  $M^{as_k}R^{bt_k}\cdots M^{as_1}R^{bt_1}$ (with $mq$ subscripts omitted). We claim that both partitions have the same characteristic polynomial and partition the same number. For the $(m^a,q^b)$-ary case, the operator $R$ multiplies by $q^b$ and $M$ adds $(m^a)^{k+1} - (m^a)^k$, where $k$ is the number of $M$ operators that occurred earlier in the sequence generating the partition. For an $(m,q)$-ary partition, the operator $R^b$ multiplies by $q^b$ and $M^a$ adds $m^{ak+a} - m^{ak}$, where $ak$ is the number of $M$ operators that occurred earlier, which is clearly a multiple of $a$. 
\end{proof}

For example, $(MR)^4$ generates a conjugate $4$-ary partition of $4^5$, and so $(M^2R^2)^4$ generates a conjugate $2$-ary partition of $2^{10}$. This method can produce conjugate $(m,q)$-ary partitions from higher powered cases. Conversely, if one could enumerate all conjugate $(m,q)$-ary partitions for a fixed pair $(m,q)$, one could enumerate them for all pairs $(m^a,q^b)$ by checking which $(m,q)$-ary partitions are generated using exponents on $M$ that are all divisible by $a$ and exponents on $R$ that are all divisible by $b$.

We close with an example that extends Example \ref{MRk}, which was critical throughout the last section when studying CMPs. In this example we present a family of CMPs, which are not self-conjugate, unlike the examples in the last section. These partitions can also be used to generate non-trivial conjugate $(m^a,m)$-ary or $(m,m^a)$-ary partitions of certain powers of m based on the last proposition.

\begin{example} Consider the CMPs generated by $(M^aR^a)^{m-1}M^aR$. Note that when $a=1$, this sequence is the same as Example \ref{MRk} and thus partitions $m^{m+1}$. We claim that these sequences generate CMPs of $m^{am+1}$ for any $a\geq 1$.
 
Consider the sequence $(M^aR^a)^{k-1}M^aR$ for $k\geq 1$.
We start with $M^a R$ generating a CMP of $m + m^a -1$. Next,  $(M^aR^a)M^aR$ generates a partition of
$$
m^a(m+m^a-1) + m^{2a}-m^a = 2m^{2a}+m^{a+1}-2m^a.
$$
By induction on $k$ one can see that $(M^aR^a)^{k-1}M^aR$ generates a CMP of 
$$km^{ka}+m^{(k-1)a+1}-km^{(k-1)a}$$
 and so
$(M^aR^a)^{m-1}M^aR$ generates a CMP of 
$$
m\cdot m^{ma}+m^{(m-1)a+1}-m\cdot m^{(m-1)a} = m^{am+1},
$$
as claimed. By conjugation,  $MR^a(M^aR^a)^{m-1}$ also generates a CMP of $m^{am+1}$. 

Since all powers of $M$ in $(M^aR^a)^{m-1}M^aR$ are divisible by $a$, we also know that $(MR^a)^{m-1}MR$ generates a conjugate $(m^a,m)$-ary partition of $m^{am+1}$. Similarly, since all powers of $R$ in $MR^a(M^aR^a)^{m-1}$ are divisible by $a$, we know that $MR(M^a R)^{m-1}$ generates a conjugate $(m,m^a)$-ary partition of $m^{am+1}$. These two partitions are, of course, conjugates of each other.
\end{example}

\section{Conclusion}

In this work we have classified all conjugate $(m,q)$-ary partitions in terms of the generating operators $M$ and $R$ (or $S$ and $J$) and in terms of a characterizing polynomial formula. These ideas led to new examples of such partitions, how they relate to their conjugates, and better understanding of self-conjugate $(m,m)$-ary partitions. Despite our new understanding, there are still many open questions and avenues for future work. The modular arithmetic involved in trying to classify which integers $n$ can have a conjugate $(m,q)$-ary partition for any given values of $m$ and $q$ seems very subtle and challenging. Perhaps studying the density of such integers $n$ among all integers will be more tractable. Even in the restricted families of self-conjugate $(m,m)$-ary partitions similar questions remain when $n$ is a power of $m$ as seen at the end of Section 5. More progress is likely to be made through a combination of numerical experimentation on a computer combined with the theoretical tools we have developed related to operators and the characteristic polynomial.

\textbf{Acknowledgments}
The authors wish to thank the anonymous referee for comments that improved the exposition of this paper, especially in Section 3. The referee also suggested asking the question about the density of integers that have a conjugate $(m,q)$-ary partition that we posed at the end of Section 4.

\end{document}